\documentclass[a4paper,11pt]{amsart} 
\usepackage{amssymb, amsmath}
\usepackage{amscd}
\numberwithin{equation}{section}
\usepackage{epsfig}
\usepackage{amsmath}
\usepackage{amsfonts,amssymb,amsopn}

\usepackage{amsthm}
\usepackage{hyperref}
\usepackage{verbatim}
\usepackage{color}
\usepackage[color,all]{xy}

\newcommand{\cE}{{\mathcal E}}
\newcommand{\cF}{{\mathcal F}}

\newcommand{\cL}{{\mathcal L}}

\newcommand{\cO}{{\mathcal O}}
\newcommand{\cQ}{{\mathcal Q}}

\newcommand{\cT}{{\mathcal T}}
\newcommand{\cU}{{\mathcal U}}
\newcommand{\cV}{{\mathcal V}}
\newcommand{\cc}{{\mathbb C}}
\newcommand{\pp}{{\mathbb P}}

\newcommand{\hY}{{\widehat{Y}}}
\newcommand{\Z}{{\mathbb Z}}

\newcommand{\tE}{\widetilde{E}}
\newcommand{\tV}{\widetilde{V}}

\newcommand{\be}{\bar{e}}
\newcommand{\bt}{\bar{t}}

\newcommand{\bE}{\overline{E}}

\newcommand{\Sym}{\mathrm{Sym}}
\newcommand{\End}{\mathrm{End}\,}
\newcommand{\Ker}{\mathrm{Ker}}
\newcommand{\Aut}{\mathrm{Aut}\,}
\newcommand{\Image}{\mathrm{Im}\,}
\newcommand{\Iden}{\mathrm{Id}}
\newcommand{\Hom}{\mathrm{Hom}}
\newcommand{\rank}{\mathrm{rk}\,}
\newcommand{\Pic}{\mathrm{Pic}}
\newcommand{\ev}{\mathrm{ev}}

\newcommand{\Kc}{K_{C}}
\newcommand{\Oc}{{{\cO}_{C}}}
\newcommand{\sRat}{\underline{\mathrm{Rat}}\,}
\newcommand{\Rat}{\mathrm{Rat}\,}
\newcommand{\sPrin}{\mathrm{\underline{Prin}}\,}
\newcommand{\Prin}{\mathrm{Prin}\,}

\newcommand{\Sec}{\mathrm{Sec}}
\newcommand{\Gr}{\mathrm{Gr}}

\newcommand{\OG}{\mathrm{OG}}
\newcommand{\Quot}{\mathrm{Quot}}
\newcommand{\cElm}{\mathcal{E}lm}
\newcommand{\cQuot}{\mathcal{Q}}
\newcommand{\Supp}{\mathrm{Supp}}
\newcommand{\isom}{\xrightarrow{\sim}}
\newcommand{\Elm}{\mathrm{Elm}}

\newcommand{\SO}{\mathrm{SO}}

\newcommand{\Spin}{\mathrm{Spin}}

\newcommand{\IQ}{\mathrm{IQ}}
\newcommand{\IQo}{\IQ^{\circ}}
\newcommand{\IQe}{\IQ_e (V)}
\newcommand{\IQeo}{\IQ^\circ_e (V)}

\newcommand{\bIQeo}{\overline{\IQ^\circ_e (V)}}
\newcommand{\LQ}{\mathrm{LQ}}

\newcommand{\dpr}{{\delta^\prime}}
\newcommand{\wj}{{\widetilde{j}}}
\newcommand{\bcE}{\overline{\cE}}
\newcommand{\eVfd}{e ( V, f, \delta )}
\newcommand{\QF}{Q_F^e}
\newcommand{\QFo}{\left( \QF \right)^\circ}

\newtheorem{theorem}{{\textbf Theorem}}[section]
\newtheorem{proposition}[theorem]{{\textbf Proposition}}
\newtheorem{corollary}[theorem]{{\textbf Corollary}}
\newtheorem{lemma}[theorem]{{\textbf Lemma}}

\newtheorem{defn}[theorem]{{\textbf Definition}}

\newtheorem{remit}[theorem]{{\textbf Remark}}
\newenvironment{remark}{\begin{remit}\rm}{\end{remit}}
\newenvironment{definition}{\begin{defn}\rm}{\end{defn}}

\newcommand{\red}[1]{{{\textcolor{red}{ #1}}}}

\title[Isotropic Quot schemes of orthogonal bundles]{Isotropic Quot schemes of orthogonal bundles over a curve}

\author{Daewoong Cheong}
\address{Chungbuk National University, Department of Mathematics, Chungdae-ro 1, Seowon-Gu, Cheongju City, Chungbuk 28644, Korea}
\email{daewoongc@chungbuk.ac.kr}

\author{Insong Choe}
\address{Department of Mathematics, Konkuk University, 1 Hwayang-dong, Gwangjin-Gu, Seoul 143-701, Korea}
\email{ischoe@konkuk.ac.kr}

\author{George H.\ Hitching}
\address{Oslo Metropolitan University, Postboks 4, St. Olavs plass, 0130 Oslo, Norway}
\email{gehahi@oslomet.no}

\subjclass[2010]{14H60; 14M17}

\keywords{Orthogonal vector bundle, curve, isotropic Quot scheme}

\begin{document}

\begin{abstract}
We study the isotropic Quot schemes $\IQe$ parameterizing degree $e$ isotropic subsheaves of maximal rank of an orthogonal bundle $V$ over a curve. The scheme $\IQe$ contains a compactification of the space $\IQeo$ of degree $e$ maximal isotropic subbundles, but behaves quite differently from the classical Quot scheme, and the Lagrangian Quot scheme in \cite{CCH1}. We observe that for certain topological types of $V$, the scheme $\IQe$ is empty for all $e$. In the remaining cases, for infinitely many $e$ there are irreducible components of $\IQe$ consisting entirely of nonsaturated subsheaves, and so $\IQe$ is strictly larger than the closure of $\IQeo$. As our main result, we prove that for any orthogonal bundle $V$ and for $e \ll 0$, the closure $\overline{\IQeo}$ of $\IQeo$ is either empty or consists of one or two irreducible connected components, depending on $\deg (V)$ and $e$. In so doing, we also characterize the nonsaturated part of $\overline{\IQeo}$ when $V$ has even rank. \end{abstract}

\maketitle

\section{Introduction}

Let $C$ be a smooth projective curve of genus $g \ge 2$ defined over $\cc$, and $V$ a vector bundle over $C$. We denote by $\Quot_{n, d} (V)$ the Quot scheme parameterizing subsheaves of $V$ of rank $n$ and degree $d$. It contains an open subscheme $\Quot_{n, d}^\circ (V)$ consisting of subbundles. If $V$ is general (for example, if $V$ is very stable), then for all $n$ and $d$, all components of $\Quot_{n, d} (V)$ are smooth and of the expected dimension, and $\Quot_{n, d}^\circ (V)$ is dense in $\Quot_{n, d} (V)$. For arbitrary $V$, however, $\Quot_{n, d} (V)$ may exhibit irregular behavior; some components may have larger dimension than expected, and/or consist entirely of nonsaturated subsheaves. The following fundamental result of Popa and Roth \cite[Theorem 6.2 and Theorem 6.4]{PR} shows that such irregularities disappear when $d$ is sufficiently small.

\begin{theorem} \label{PRirr}  For $0 < n < \rank (V)$, there is an integer $d_n(V)$ such that if $d \le d_n(V)$, the Quot scheme $\Quot_{n, d}(V)$ is irreducible and of the expected dimension, and a general point corresponds to a subbundle of $V$ which is a stable vector bundle. \end{theorem}

In \cite{CCH1}, an analogue of Theorem \ref{PRirr} was proven for symplectic bundles. A bundle $W$ over $C$ is called \textsl{$L$-valued symplectic} if it admits a bilinear nondegenerate antisymmetric form $\sigma \colon W \otimes W \to L$ for some line bundle $L$. A symplectic bundle $W$ always has even rank $2n$.  A subsheaf $E \subset W$ is said to be \textsl{isotropic} if $\sigma |_{E \otimes E} \equiv 0$. An isotropic subbundle (resp., isotropic subsheaf) of maximal rank $\frac{1}{2} \rank (W)$ is called a \textsl{Lagrangian subbundle} (resp., \textsl{Lagrangian subsheaf}). We denote by $\LQ_e (W)$ the subscheme of $\Quot_{n, e}(W)$ parameterizing Lagrangian subsheaves.
 
\begin{theorem} \label{LagQuot} (\cite[Theorem 4.1]{CCH1}) For any symplectic bundle $W$ over $C$, there is an integer $e(W)$ such that for each $e \le e(W)$, the {Lagrangian Quot scheme} $\LQ_e (W)$ is irreducible and of the expected dimension, and a general point corresponds to a subbundle which is a stable vector bundle. \end{theorem} 

The goal of the present article is to prove another analogue of Theorem \ref{PRirr}, for \emph{orthogonal bundles}. An orthogonal bundle $V$ is defined in the same way as a symplectic one, except that the form $\sigma$ is \emph{symmetric} instead of antisymmetric. Note that an orthogonal bundle may have odd rank. Again, any isotropic subsheaf has rank at most $\frac{1}{2} \rank (V)$.

\begin{definition} \label{defnIQe} Let $V$ be an $L$-valued orthogonal bundle of rank $r = 2n$ or $2n+1$. For each integer $e$, we define the \textsl{isotropic Quot scheme} $\IQe$ by
\[ \IQe \ := \ \left\{ [j \colon E \to V] : E \hbox{ isotropic of rank $n$ in } V \right\} \ \subseteq \ \Quot_{n, e} (V) . \]
The same argument as in \cite[Lemma 2.2]{CCH1} shows that $\IQe$ is a closed subscheme of $\Quot_{n, e}(V)$. 
 We denote by $\IQeo$ the open subscheme consisting of saturated isotropic subsheaves; that is, isotropic subbundles. \end{definition}

The symmetric form $\sigma$ induces isomorphisms $V \isom V^* \otimes L$ and $\det (V)^2 \isom L^r$. 
 In the symplectic case, where $r = 2n$, we always have $\det V \cong L^n$ (see \cite{BGsymp}). However, if $V$ is $L$-valued orthogonal of even rank $2n$ then $\det (V)$ can be any square root of $L^{2n}$. It emerges that $V$ has an isotropic subbundle of rank $n$ if and only if $\det (V) \cong L^n$ (Lemma \ref{EvenRankExistence}). In contrast, an orthogonal bundle of odd rank $2n+1$ always has an isotropic subbundle of rank $n$ (Lemma \ref{OddRankExistence}).

When nonempty, the schemes $\IQe$ still exhibit more complicated behavior than $\Quot_{n,e}(V)$ and $\LQ_e(V)$. In {\S}\:\ref{EvenRankComponents} and {\S}\:\ref{OddRankComponents}, we show that $\IQe$ has multiple components consisting entirely of nonsaturated subsheaves, and that for $r = 2n$ it is not equidimensional, even for arbitrarily small $e$. We indicate three reasons for this behavior in the even rank case which do not apply to $\Quot_{n, e} (V)$ or $\LQ_e (W)$.

\begin{enumerate}
\item[(i)] If $V$ admits a rank $n$ isotropic subbundle, then the orthogonal Grassmann bundle $\OG(n, V)$ has two connected components (Proposition \ref{OGVTwoCpts}). When there are sections of both components defining subbundles of the same degree, $\IQeo$ cannot be connected. (See \cite[Theorem 5.3]{CH5}.)

\item[(ii)] Any two isotropic subbundles defining sections of the same component 
 of $\OG(n, V)$ have degrees of the same parity (Theorem \ref{ParityComponents}). 
 Thus a fixed component $\OG(n, V)_\delta$ admits no sections corresponding to isotropic subbundles of degree $e$ if $e$ does not have the appropriate parity. (In fact, if $\deg (L)$ is even, by \cite[Theorem 1.2 (1)]{CH3} the locus $\IQeo$ of saturated subsheaves is empty for infinitely many $e$.) However, by taking full rank subsheaves of a rank $n$ isotropic subbundle, we can produce \emph{nonsaturated} isotropic subsheaves in $\IQe$ whose saturations are sections of either component of $\OG(n, V)$. It follows that for infinitely many $e$, there are components of $\IQe$ consisting entirely of nonsaturated isotropic subsheaves.

\item[(iii)] For fixed $g$, the expected dimension of a component of $\IQeo$ is given by
\[  I (n, \ell, e)  := -(n-1)e  -\frac{n(n-1)}{2} (g-1-\ell) \]
where $\ell := \deg(L)$ (Proposition \ref{ExpDimSmoothness}). On the other hand, the locus of nonsaturated subsheaves $E$ of degree $e$ which are contained in a fixed rank $n$ isotropic subbundle $\bE$ of degree $\be$ is parameterized by $\Quot_{0, e}(\overline{E})$, which has dimension $n(\be - e)$. Hence the dimension of the nonsaturated locus of $\IQe$ is expected to exceed that of the saturated locus for $e \ll 0$. (For comparison; for a symplectic bundle $W$, the expected dimension of a component of $\LQ^\circ_e(W)$ is a linear function of $e$ with slope $-(n+1)$; see \cite[Proposition 2.4]{CCH1}.)
\end{enumerate}

Thus for rank $2n$, we cannot expect $\IQe$ to be irreducible or equidimensional, even for $e \ll 0$. However, the closure $\overline{\IQeo}$ of $\IQeo$ in $\IQe$ is better behaved. Let $w(V) \in \Z_2$ be the Stiefel--Whitney class of $V$, which will be discussed in {\S}\:\ref{StWh}. The following are the main results of this paper.

\begin{theorem} \label{Even} Let $V$ be an $L$-valued orthogonal bundle of even rank $2n \ge 4$ and of determinant $L^n$. Then there is an integer $e(V)$ such that:
\begin{enumerate}
\item[(a)] If $\ell$ is even, then for each $e \le e(V)$ with $e \equiv w (V) \mod 2$, the locus $\overline{\IQeo}$ has two nonempty connected irreducible components, both of which are generically smooth of dimension $I(n, \ell, e)$. Moreover, $\IQeo$ is empty when $e \not \equiv w (V) \mod 2$.
\item[(b)] If $\ell$ is odd, then for each $e \le e(V)$, the locus $\overline{\IQeo}$ is nonempty and irreducible, and generically smooth of dimension $I(n, \ell, e)$.
\end{enumerate} \end{theorem}

\begin{theorem} \label{Odd} Let $V$ be an $L$-valued orthogonal bundle of rank $2n+1 \ge 3$. There is an integer $e(V)$ such that for each $e \le e(V)$ with $e \equiv w(V) \mod 2$, the locus $\overline{\IQeo}$ is nonempty, irreducible and generically smooth of dimension $I(n+1, \ell, e - \frac{\ell}{2})$. \end{theorem}

In proving these results, we find a criterion for a nonsaturated isotropic subsheaf {of an orthogonal $V$ of even rank} to occur as a limit of isotropic subbundles. To state this, we make a definition.

\begin{definition} \label{TorsionTypeT}
\quad \begin{enumerate}
\item[(a)] We say that a torsion sheaf $\tau$ on $C$ \textsl{is of type} $\cT$ if there is a filtration
\[
0 = \tau_0 \:\subset \:\tau_1 \:\subset \:\cdots \: \subset \:\tau_k =\tau,
\]
where $\tau_i/\tau_{i-1} \cong \mathcal{O}_{x_i} \otimes \mathbb{C}^2$ for some $x_i \in C$ for $1 \le i \le k$.
\item[(b)] We say that an element $[E \to V]$ of $\IQe$ is \textsl{of type} $\cT$ if either $E \in \IQeo$ or the quotient $\bE/E$ is of type $\cT$, where $\bE$ is the saturation of $E$. \end{enumerate} \end{definition}We show the following.

\begin{theorem} \label{TypeT} Let $V$ be an orthogonal bundle {of even rank $2n \ge 4$}.
\begin{enumerate}
\item[(a)] If a point $[E \to V]$ of $\IQe$ lies in the closure of $\IQeo$, then it is of type $\cT$.
\item[(b)] There is an integer $e(V)$ such that for $e \le e(V)$, every point $[E \to V]$ in $\IQe$ of type $\cT$ lies in the closure of $\IQeo$.
\end{enumerate}
\end{theorem}
Therefore by Theorem \ref{Even}, the space $\IQ_e^\cT (V)$ of subsheaves of type $\cT$ coincides with $\overline{\IQeo}$ for $e \le e(V)$, and provides a compactification of $\IQeo$ which behaves better than $\IQe$. In this sense, the scheme $\IQ_e^\cT (V)$ can be considered a ``correct'' analog of the Lagrangian Quot scheme $\LQ_e (W)$ of a symplectic bundle $W$. Moreover, $\overline{\IQeo}$ naturally has a modular interpretation and inherits a universal family from $Q_{n, e} (V)$, which may in some situations be an advantage compared to compactifications of $\IQeo$ given by Hilbert schemes or moduli of stable maps. In a forthcoming paper, we shall develop an intersection theory on $\overline{\IQeo}$ using Gromov--Witten invariants, as done for Lagrangian Quot schemes in \cite{CCH2}, with enumeration of maximal isotropic subbundles as an application.\\
\par
Here is a summary of the paper. In {\S}\:\ref{existence}, we characterize those orthogonal bundles of rank $2n$ admitting a rank $n$ isotropic subbundle, and study the orthogonal Grassmann bundles of such $V$. From {\S}\:\ref{RankTwo} through {\S}\:\ref{EvenRank}, we assume that $\rank (V) = 2n \ge 4$, as the statements for the odd rank case turn out to follow relatively easily from the even rank case.

In {\S}\:\ref{IsotQuotSch}, we find the expected dimension of $\IQeo$ and compute the Zariski tangent spaces $T_E \IQe$, including the case where $E$ is not saturated. In {\S}\:\ref{OrthExt}, we recall or prove some facts on orthogonal extensions and principal parts, extending results in \cite{Hit1}. In {\S}\:\ref{EvenRankComponents} we describe the totally nonsaturated components of $\IQe$ and prove Theorem \ref{TypeT} (a). In {\S}\:\ref{OrthExtII} we develop further results on orthogonal extensions, liftings and geometry in extension spaces. These are applied in {\S}\:\ref{EvenRank} to prove Theorems \ref{Even} and \ref{TypeT} (b). In {\S}\:\ref{OddRank} we discuss the odd rank case, and prove Theorem \ref{Odd}.

The proofs of some (but not all) results in {\S\S} \ref{OrthExt}, \ref{OrthExtII} and \ref{EvenRank} are virtually identical to their symplectic counterparts in \cite[{\S}\:3]{CCH1}. We indicate where this is the case, and in some cases we omit details.


\subsection*{Acknowledgements} We are grateful to the referee of an early version of \cite{CCH1} for comments which also improved the present work. The first named author was supported by Basic Science Research Program through the National Research Foundation of Korea (NRF) funded by the Ministry of Education (2018R1D1A3B07045594) and Korea Institute for Advanced Study(KIAS) grant funded by the Korea government (MSIP). The second named author was supported by Basic Science Research Programs through the National Research Foundation of Korea (NRF) funded by the Ministry of Education (NRF-2017R1D1A1B03034277). The third named author sincerely thanks Konkuk University, Hanyang University and the Korea Institute of Advanced Study for financial support and hospitality, and acknowledges gratefully a period of research leave supported by Oslo Metropolitan University in 2017-2018.

\subsection*{Notation} Throughout, $C$ denotes a complex projective smooth curve of genus $g \ge 2$. For a vector bundle $W$ over $C$, we write $W|_x$ for the fiber of $W$ at $x \in C$. For a subsheaf $E \subset W$, we denote by $\bE$ the saturation, which is a vector subbundle of $W$. We write $\Quot_{n, d} (W)$ for the Quot scheme parameterizing subsheaves $[j \colon E \to W]$ of rank $n$ and degree $d$. A point $[j \colon E \to W] \in \Quot_{n, d} (W)$ may be denoted by $[E \to W]$ or simply $j$ or $E$ to ease notation. As a special case, for $t \ge 0$, we denote by $\Elm^t (W)$ the Quot scheme $\Quot_{\rank(W), \deg(W) - t} (W)$ parameterizing elementary transformations $[W' \subset W]$ where $W/W'$ is torsion of length $t$.

\section{Orthogonal bundles and isotropic subbundles} \label{existence}

\begin{definition} A vector bundle $V \to C$ of rank $r$ is called an \textsl{$L$-valued orthogonal bundle} if there is a bilinear nondegenerate symmetric form $\sigma \colon V \otimes V \to L$ for some line bundle $L$; equivalently, if there is a symmetric isomorphism $V \isom V^* \otimes L$. \end{definition}

\subsection{Determinants of orthogonal bundles} \label{determinants}

If $V$ is an $L$-valued orthogonal bundle of rank $r$, then by the isomorphism $V \isom V^* \otimes L$, we have $\det(V)^2 = L^r$. It follows that if $r = 2n+1$, then $\deg (L)$ is even.

If $r = 2n$, then $\det (V) = \eta \otimes L^n$ for some $\eta$ of order two in $\Pic^0 (C)$, which may be nontrivial. For example,  the orthogonal direct sum $V \cong \Oc \perp \eta$ for a nontrivial $\eta$  of order two is an $\Oc$-valued orthogonal bundle with $\det (V) = \eta$. (In contrast, an $L$-valued symplectic bundle of rank $2n$ always has determinant $L^n$; see \cite[{\S}\:2]{BGsymp}.)

\begin{remark} \label{LnotUnique} Note that $V = \Oc \oplus \eta$ also admits the $\eta$-valued quadratic form
\[ \left( ( \lambda_1 , \nu_1 ) , ( \lambda_2 , \nu_2 ) \right) \ \mapsto \ \lambda_1 \nu_2 + \lambda_2 \nu_1 . \]
Thus in general $L$ depends not only on the underlying vector bundle $V$, but also on $\sigma$. \end{remark}

\begin{lemma} \label{twistEven} Let $V$ be an $L$-valued orthogonal bundle of rank $2n$ with $\det V = \eta \otimes L^n$ for some $\eta$ of order two in $\Pic^0(C)$. 
\begin{enumerate}
\item[(a)] If $\deg L$ is even, then there is a line bundle $M$ such that the twist $V \otimes M^{-1}$ is an $\Oc$-valued orthogonal bundle of determinant $\eta$.
\item[(b)] If $\deg L$ is odd, then  there is a line bundle $M$ such that the twist $V \otimes M^{-1}$ is an $\Oc (x) $-valued orthogonal bundle of determinant $\eta (nx)$ for some $x \in C$.
 \end{enumerate}
 \end{lemma}
 \begin{proof}
Choose a square root $M$ of $L$ and $L(x)$ for (a) and (b), respectively. 
 \end{proof}
\begin{lemma} \label{twistOdd} Let $V$ be an $L$-valued orthogonal bundle of rank $2n+1$. Then there is a line bundle $M$ such that the twist $V \otimes M^{-1}$ is an $\Oc$-valued orthogonal bundle of trivial determinant.
 \end{lemma}
 \begin{proof}
 The line bundle $M = L^{n+1} \otimes (\det V)^{-1}$ satisfies the desired property. 
  \end{proof}

\subsection{Isotropic subbundles of orthogonal bundles} \label{isotsubbundles}

Let $V$ be a bundle of rank $r$ with orthogonal form $\sigma \colon V \otimes V \to L$. We recall that a subsheaf $E$ of $V$ is called \textsl{isotropic} if $\sigma|_{E \otimes E} \equiv 0$. Our focus will be on isotropic subsheaves of the maximal possible rank, which by linear algebra is $\left\lfloor \frac{r}{2} \right\rfloor$. We write $n := \left\lfloor \frac{r}{2} \right\rfloor$, so $r = 2n$ or $2n + 1$. Not all orthogonal bundles of rank $2n$ have isotropic subbundles of rank $n$, and we shall now characterize those which do. Firstly, we recall the notion of an orthogonal Hecke transformations from \cite[{\S}\:3]{BG}.

Let $V$ be an $L$-valued orthogonal bundle of rank $2n$. Let $\Lambda$ be an $n$-dimensional isotropic subspace of a fiber $V|_x$. Let $\tV$ be the subsheaf of $V$ of sections with values in $\Lambda$ at $x$. This is an elementary transformation
\begin{equation} 0 \ \to \ \tV \ \to \ V \ \to \ \cO_x \otimes \cc^n \ \to \ 0 . \label{OrthHecke} \end{equation}
By \cite[Proposition 3.1]{BG}, the bundle $\tV$ is $L(-x)$-valued orthogonal, and clearly of determinant $\det (V) \otimes \Oc (-nx)$. We call $\tV$ an \textsl{orthogonal Hecke transformation of $V$}. 

\begin{lemma}  \label{EvenRankExistence}
Let $V$ be an $L$-valued orthogonal bundle of rank $2n$. Then the following statements are equivalent.
\begin{enumerate}
\item[(1)] $V$ has an isotropic subbundle of rank $n$.
\item[(2)] $\det(V) = L^n$.
\item[(3)] There is a line bundle $M$ such that either $V \otimes M^{-1}$ or an orthogonal Hecke transformation of $V \otimes M^{-1}$ is an $\Oc$-valued orthogonal bundle of trivial determinant.
\item[(4)] There exists a nonempty Zariski open subset $U$ of $C$ together with a trivialization $V|_U \cong \cc^{2n} \times U$ such that the standard basis of $\cc^{2n}$ forms an orthonormal frame.
\end{enumerate}
\end{lemma}

\begin{proof} $(1) \Rightarrow (2)$: Suppose $V$ has a rank $n$ isotropic subbundle $E$. Since $E$ coincides with its orthogonal complement $E^\perp$, there is an exact sequence $0 \to E \to V \to \Hom(E, L) \to 0$. From this it follows that $\det (V)  = L^n$.

$(2) \Rightarrow (3)$: Using Lemma \ref{twistEven}: If $\deg L$ is even, then some twist $V \otimes M^{-1}$ is an $\Oc$-valued orthogonal bundle of trivial determinant. If $\deg L$ is odd, some twist $V \otimes M^{-1}$ is $\Oc(x)$-valued orthogonal of determinant $\Oc (nx)$, for some $x \in C$. Choose any $n$-dimensional isotropic subspace $\Lambda \subset (V \otimes M)|_x$ and consider the associated orthogonal Hecke transformation
\[ 0 \ \to \ \widetilde{V \otimes M^{-1}} \ \to \ V \otimes M^{-1} \ \to \ \cO_x \otimes \cc^n \ \to \ 0 . \]
Then $\widetilde{V \otimes M^{-1}}$ is $\Oc$-valued orthogonal of trivial determinant.

$(3) \Rightarrow (4)$: Let $V_0$ be an $\Oc$-valued orthogonal bundle of trivial determinant obtained from $V$ by tensor product and/or Hecke transformations as stated. Then we can find a principal $\SO_{2n}$-bundle $P_0$ whose associated vector bundle is $V_0$. (In general, there are two choices of $P_0$, by \cite{Serm12}.) As $\dim (C) = 1$ and $\SO_{2n}$ is connected, $P_0$ is Zariski locally trivial (see \cite[2.11]{RamI}). In particular,  we can find 
 an orthonormal frame for $V_0|_U$ for some Zariski open subset $U \subseteq C$. Then for any line bundle $M$, any orthogonal Hecke transform $V_1$ of $V_0 \otimes M^{-1}$ inherits the orthonormal frame over $U \setminus B$ for some finite set $B$. In particular, $V$ admits an orthonormal frame over a Zariski open subset of $C$.

$(4) \Rightarrow (1)$: Given an orthonormal frame over a Zariski open subset $U \subseteq C$, we can choose a rank $n$ isotropic subbundle $E_U \subset V|_U$. As $\dim (C) = 1$, we can extend $E_U$ to a rank $n$ isotropic subbundle $E \subset V$. \end{proof}

\begin{remark} This implies that if $V$ is $L$-valued orthogonal of even rank and  $\det (V) = \eta \otimes L^n$ for $\eta \neq \Oc$, then $V$ has \emph{no} rank $n$ isotropic subbundle. Let us see directly why this slightly surprising fact holds in a special case. We consider again the orthogonal direct sum $V = \Oc \perp \eta$ for a nontrivial $\eta$ of order two, with $\Oc$-valued quadratic form given by
\[ \left( ( \lambda_1, \nu_1 ) , ( \lambda_2 , \nu_2 ) \right) \ \mapsto \ \lambda_1 \lambda_2 + \nu_1 \otimes \nu_2 . \]
The locus of isotropic lines in $\pp V$ is given by $\{ ( 1 : \nu/\lambda ) : ( \nu / \lambda )^{\otimes 2} \ = \ -1 \}$. 
Clearly this is isomorphic to the \'etale double cover $C_\eta \to C$ determined by $\eta$. But since $\eta$ is nontrivial, $C_\eta$ is irreducible, so the cover admits no section over any Zariski open subset of $C$. It follows that $\Oc \perp \eta$ has no line subbundle isotropic with respect to this choice of quadratic form. \end{remark}

The next two results show that isotropic subbundles exist in all remaining cases.

\begin{lemma}  \label{OddRankExistence}
Every orthogonal bundle $V$ of  odd rank $2n+1 \ge 1$ admits an isotropic subbundle of rank $n$.
\end{lemma}
\begin{proof}
By Lemma \ref{twistOdd}, there is a line bundle $M$ such that the twist $V \otimes M^{-1}$ is $\Oc$-valued orthogonal bundle of trivial determinant. As in the case of even rank, the associated $\SO_{2n+1}$-bundle is Zariski locally trivial and we can find an isotropic subbundle  by choosing a trivial isotropic subbundle over a Zariski open subset and extending it to $C$.
\end{proof}

\begin{corollary} Every orthogonal bundle of rank $2n \ge 4$ admits an isotropic subbundle of rank $n-1$. \end{corollary}

\begin{proof} Let $V$ be $L$-valued orthogonal of rank $2n$ and determinant $\eta \otimes L^n$. As the property of admitting an isotropic subbundle of rank $n-1$ is preserved by Hecke transforms, we may assume that $\deg (L)$ is even. Let $M$ be a square root of $L$. Then the orthogonal direct sum $V \perp M$ is $L$-valued orthogonal of rank $2n+1$. By Lemma \ref{OddRankExistence}, we can find a rank $n$ isotropic subbundle $F \subset V \perp M$. Clearly $F \cap V$ contains an isotropic subbundle of rank $n-1$ in $V$. \end{proof}

\subsection{Stiefel--Whitney classes} \label{StWh}

Let $V$ be an $L$-valued orthogonal bundle of rank $r = 2n$ or $2n+1$. If $L = \Oc$ and $\det(V) = \Oc$, then $V$ determines a principal $\SO_r$-bundle over $C$ (possibly up to a choice of orientation; see \cite{Serm12}). The \textsl{Stiefel--Whitney class} $w_2 ( V ) \in H^2 (C, \Z_2 )$ is the obstruction to lifting this $\SO_r$-bundle to a $\Spin_r$-bundle. See \cite[{\S}\:1]{Bea}.

Suppose now that $\ell = \deg(L)$ is even. If $r = 2n$, suppose in addition that $\det(V) = L^n$. Then by Lemma \ref{twistEven} and \ref{twistOdd} there exists $M$ such that $V \otimes M^{-1}$ is $\Oc$-valued orthogonal of trivial determinant. 

\begin{definition} Under the above assumptions, we define 
\[ w(V) \ := \ w_2 (V \otimes M^{-1}) + \frac{n \ell}{2} \  \ \text{in} \  H^2 (C, \Z_2 ) \cong \Z_2. \] 
 \end{definition}

\noindent Clearly this is independent of the choice of $M$. It will emerge in Theorem \ref{ParityComponents} (a) and Theorem \ref{Odd} that $w(V)$ is the parity of the degree of any rank $n$ isotropic subbundle of $V$.

\subsection{Orthogonal bundles of rank two} \label{RankTwo}

Here we characterize $L$-valued orthogonal bundles of rank two admitting an isotropic line subbundle. By Lemma \ref{EvenRankExistence}, such a bundle must have determinant $L$.

\begin{proposition} \label{rank2ortho} Let $V$ be a vector bundle of rank two with $\det(V) = L$. Then $V$ has an $L$-valued orthogonal structure if and only if $V $ is a direct sum of two line subbundles. Also in this case, the direct summands are isotropic subbundles.
\end{proposition}

\begin{proof} Suppose $V$ is $L$-valued orthogonal of rank two. For each $x \in C$, there are exactly two points of $\pp V$ corresponding to isotropic lines in $V|_x$. Thus we get an \'etale double cover $\pi \colon \widetilde{C} \to C$. Since $\det(V) = L$, by Lemma \ref{EvenRankExistence} there is an isotropic line subbundle. This gives a section of $\pi$, so $\pi$ must be a split double cover. It follows that $V$ is the direct sum of two isotropic line subbundles.

Conversely, a decomposable bundle $V$ of rank two and determinant $L$ is necessarily of the form $N \oplus N^{-1} L$. Such a bundle admits the $L$-valued orthogonal form
\[ \left( ( v_1 , w_1 ) , ( v_2 , w_2 ) \right) \ \mapsto \ v_1 \otimes w_2 + v_2 \otimes w_1 , \]
with respect to which $N$ and $N^{-1}L$ are isotropic.
\end{proof}

In view of Proposition \ref{rank2ortho}, it is easy to describe the isotropic Quot schemes of an orthogonal bundle of rank two. Thus we shall henceforth assume $r \ge 3$. Moreover, as the results for the odd rank case will be derived from the even rank case in {\S}\:\ref{OddRank}, \textbf{we shall assume that $V$ has rank $2n \ge 4$ until {\S}\:\ref{OddRank}}.

\subsection{Orthogonal Grassmann bundles}

Let $\sigma$ be a nondegenerate symmetric bilinear form on $\cc^{2n}$. The \textsl{orthogonal Grassmannian} $\OG ( n , 2n )$ parameterizes linear subspaces of dimension $n$ in $\cc^{2n}$ which are isotropic with respect to $\sigma$. More geometrically; $\sigma$ defines a smooth quadric hypersurface in $\pp^{2n-1}$, and $\OG (n, 2n)$ parameterizes projective linear subspaces of dimension $n-1$ contained in this quadric. The following well known fact will be used several times.

\begin{lemma} \label{GH} The orthogonal Grassmannian $\OG(n, 2n)$ has two irreducible connected components. Two $n$-dimensional isotropic subspaces $\Lambda_1$ and $\Lambda_2$ belong to the same component if and only if $\dim ( \Lambda_1 \cap \Lambda_2 ) \equiv n \mod 2$. \end{lemma}

\begin{proof} This is a restatement of \cite[Proposition 2, p.\ 735]{GH}. \end{proof}

For an $L$-valued orthogonal bundle $V$ of rank $2n$, we denote by $\OG (V)$ the closed subvariety of the Grassmann bundle $\Gr ( n, V )$ whose fiber at $x \in C$ is the orthogonal Grassmannian $\OG(n, V|_x) \subset \Gr(n, V|_x)$. 

\begin{proposition} \label{OGVTwoCpts} Let $V$ be an $L$-valued orthogonal bundle of rank $2n$ satisfying the equivalent conditions of Lemma \ref{EvenRankExistence}. Then $\OG (V)$ has two connected components. \end{proposition}

\begin{proof} By hypothesis, $V$ has a rank $n$ isotropic subbundle. Let $\pi \colon \OG (V) \to C$ be the projection. By Lemma \ref{GH}, for each $x \in C$ the fiber $\pi^{-1} (x)$ has two connected components. A rank $n$ isotropic subbundle gives a section of the fiber bundle $\pi$. Hence there are two connected components of $\OG(V)$, one of which contains that section and the other not.
\end{proof}

\begin{theorem} \label{ParityComponents} Let $V$ be as in Proposition \ref{OGVTwoCpts}.
\begin{enumerate}
\item[(a)] Suppose $\deg(L)$ is even. Then a section belonging to either component of $\OG(V)$ defines a rank $n$ isotropic subbundle of $V$ whose degree has the same parity as $w(V)$.
\item[(b)] Suppose $\deg(L)$ is odd. Then two rank $n$ isotropic subbundles $E$ and $E'$ of $V$ define sections of the same component of $\OG(V)$ if and only if $\deg (E)$ and $\deg (E')$ have the same parity.
\end{enumerate} \end{theorem}

\begin{proof} 
(a) 
By Lemma \ref{twistEven} it suffices to consider the case where $L = \Oc$ and $\det (V) = \Oc$. This is proven in \cite[Theorem 1.2 (1)]{CH3}.

(b) By Lemma \ref{twistEven} we may assume $L = \Oc(x)$ and $\det (V) = \Oc(nx)$ for some $x \in C$. Choose an $n$-dimensional isotropic subspace $\Lambda \subset V|_x$ and let $0 \to \tV \to V \to \cO_x \otimes \cc^n \to 0$ be the associated orthogonal Hecke transformation. Fix $\delta \in \{ 1 , 2 \}$. If $E \subset V$ is a rank $n$ isotropic subbundle with $\deg(E) \equiv \delta \mod 2$, we have a diagram
\[ \xymatrix{ 0 \ar[r] & \tV \ar[r] & V \ar[r] & \cc^n_x \ar[r] & 0 \\
0 \ar[r] & \tE \ar[r] \ar[u] & E \ar[r] \ar[u] & \cc^k_x \ar[r] \ar[u] & 0 } \]
where $k = n - \dim ( E|_x \cap \Lambda )$. Thus $\deg(\tE) \ = \ \deg(E) - n + \dim ( E|_x \cap \Lambda )$, so
\[ \dim ( E|_x \cap \Lambda ) \ \equiv \ \deg ( \tE ) + \delta + n \mod 2 . \]
Now since $\tV$ is $\Oc$-valued orthogonal, by part (a) the number $\deg(\tE)$ is constant modulo 2. Hence the expression on the right is constant modulo $2$ as $E$ varies over the rank $n$ isotropic subbundles of $V$ with $\deg (E) \equiv \delta \mod 2$. By Lemma \ref{GH}, the fibers $E|_x$ of all such $E$ belong to the same component of $\OG(V)|_x$. Therefore, all such $E$ define sections of the same component of $\OG(V)$. \end{proof}

\begin{definition} \label{labeling} We label the components $\OG(V)_\delta$, where $\delta \in \{ 1 , 2 \}$. When $\deg(L)$ is even, we choose an arbitrary labeling, just to distinguish one from the other. When $\deg(L)$ is odd, we denote by $\OG(V)_\delta$ the component containing the subbundles $E$ with $\deg (E) \equiv \delta \mod 2$ (cf.\ Theorem \ref{ParityComponents} (b)). \end{definition}

In view of Lemma \ref{EvenRankExistence}, \textbf{for the duration of sections 3-7, we will assume that $\det(V) = L^n$ whenever we consider an $L$-valued orthogonal bundle $V$ of rank $2n$}.

\section{Isotropic Quot schemes} \label{IsotQuotSch}

Let $V$ be an $L$-valued orthogonal bundle of rank $2n$. We shall now study the isotropic Quot schemes $\IQe$ (Definition \ref{defnIQe}). We begin with another definition.

\begin{definition} For $\delta \in \{ 1, 2 \}$ we write $\IQe_\delta$ (respectively, $\IQeo_\delta$) for the locus in $\IQe$ of subsheaves (respectively, subbundles) whose saturations define sections of $\OG(V)_\delta$. \end{definition}

\noindent Clearly $\IQe = \IQe_1 \: \sqcup \: \IQe_2$, but still each $\IQe_\delta$ may have several components.

\subsection{Tangent spaces of isotropic Quot schemes}

We now describe the Zariski tangent space to the isotropic Quot scheme $\IQe$ at a point $[E \to V]$ corresponding to a rank $n$ isotropic subsheaf $E \subset V$ which need not be saturated.
 Let $\bE$ be the saturation of $E$ in $V$. As $\bE = \bE^\perp$, the quotient $V / \bE$ is isomorphic to $\bE^* \otimes L$. Then the torsion sheaf $T := \bE / E$ coincides with the torsion subsheaf of $V/E$. 
 The dual elementary transformation gives a sequence $0 \to \bE^* \otimes L \to E^* \otimes L \to T' \to 0$, where $T' = {\mathcal Ext}^1 (T, L)$ is (noncanonically) isomorphic to $T$. This gives rise to a diagram
\begin{equation} \xymatrix{ H^0 ( C, \Hom (E, T )) \ar@{^{(}->}[r] & H^0 ( C, \Hom (E, V/E)) \ar@{->>}[r] \ar[dr]_c & H^0 ( C, \Hom ( E, \bE^* \otimes L ) ) \ar@{^{(}->}[d] \\
 & & H^0 ( C, \Hom (E, E^* \otimes L )) \ar[d] \\
 & & H^0 ( C, \Hom (E, T')) }
\label{TangSpaceDiag} \end{equation}
We have now a generalization of \cite[Lemma 4.3]{CH3}.

\begin{lemma} \label{Zariski} Let $V$ and $E$ be as above. The Zariski tangent space to $\IQe$ at $[ E \to V ]$ is the inverse image of $H^0 ( C, \wedge^2 E^* \otimes L)$ by the map $c$ in (\ref{TangSpaceDiag}). In particular, if $E$ is saturated then $T_E \IQe \cong H^0 (C, \wedge^2 E^* \otimes L)$. \end{lemma}

\begin{proof} Let $\alpha \colon E \to V/E \cong T \oplus (\bE^* \otimes L)$ represent a tangent vector to $\Quot_{n, e}(V)$ at $[E \to V]$. Let $U$ be a nonempty open subset of $C$ containing no points of $\Supp (T)$, so that $E|_U = \bE|_U$ as subbundles of $V|_U$. Shrinking $U$ if necessary, we may also assume that $L|_U$ is trivial. Then $V|_U$ is an extension
\[ 0 \ \to \ E|_U \ \to \ V|_U \ \to \ E^*|_U \ \to \ 0 \]
and $\alpha|_U$ defines a first-order deformation of $\left[ E|_U \to V|_U \right]$.

For each $x \in U$, the section $\alpha$ defines an element $\alpha(x) \in T_{(E|_x)} \Gr(n, V|_x)$, and the deformation preserves isotropy of $E|_U$ if and only if
\[ \alpha(x) \ \in \ T_{(E|_x)} \OG ( n, V|_x ) \ \subset \ T_{(E|_x)} \Gr (n, V|_x) , \]
for all $x \in U$; that is,
\begin{equation} \alpha|_U \ \in \ H^0 ( U , \wedge^2 E^*|_U ) \ \subset \ H^0 ( U, E^* \otimes E^*|_U) . \label{alphaU} \end{equation}
As $U$ contains no points of $\Supp (T)$, the restriction $\alpha|_U$ is canonically identified with $c(\alpha)|_U$ in (\ref{TangSpaceDiag}). Therefore, since $U$ is dense, (\ref{alphaU}) is equivalent to $c(\alpha) \in H^0 ( C, \wedge^2 E^* \otimes L )$. \end{proof}

Now let $E$ be any bundle of rank $n$ and degree $e$, and $L$ a line bundle of degree $\ell$. To ease notation, we set
\[ I(n, \ell, e) \ := \ \chi ( C, \wedge^2 E^* \otimes L ) \ = \ -(n-1)e - \frac{1}{2}n(n-1)(g-1-\ell) . \]

\begin{proposition} \label{ExpDimSmoothness} Let $V$ be an $L$-valued orthogonal bundle of rank $2n$.
\begin{enumerate}
\item[(a)] Any component of $\IQeo$ has dimension at least $I(n, \ell, e)$.
\item[(b)] For any point $[j \colon E \to V] \in \IQeo$, if $h^1 (C, L \otimes \wedge^2 E^*) = 0$ then $\IQeo$ is smooth of dimension $I(n, \ell, e)$ at $j$. 
\end{enumerate} 
\end{proposition}

\begin{proof} Using Lemma \ref{Zariski}, parts (a) and (b) respectively are proven identically to \cite[Proposition 2.4 (a) and (c)]{CCH1}. \end{proof}

Although we do not rely on this fact later, we would like to point out that the locus $\IQeo$ behaves nicely if $V$ is general in moduli. For simplicity, we restrict to the case of trivial determinant.

\begin{lemma} \label{VeryStableSmooth} Let $V$ be a stable $\Oc$-valued orthogonal bundle of rank $2n \ge 4$ and trivial determinant. If $V$ is general in moduli, then $h^1 ( C, \wedge^2 E^* ) = 0$ for all rank $n$ isotropic subbundles $E$ of $V$. In particular, $\IQeo$ is empty if $I(n, \ell, e) < 0$, and smooth of dimension $I(n, \ell, e)$ otherwise. \end{lemma}

\begin{proof} We follow \cite[p.\ 227]{BR}. Suppose $E \subset V$ is a rank $n$ isotropic subbundle. Let $P$ be a principal $\SO_{2n}$-bundle whose associated vector bundle via the standard representation on $\cc^{2n}$ is $V$. The subbundle $E$ defines a reduction of structure group of $P$ to a maximal parabolic subgroup $Q$ stabilizing a fixed isotropic subspace of dimension $n$ in $\cc^{2n}$. A Lie algebra computation shows that $\wedge^2 E$ is the associated bundle of the nilpotent radical of $\mathrm{Lie} (Q)$. Thus, if $h^1 (C, \wedge^2 E^* ) \ne 0$, then by Serre duality, $h^0 ( C, \Kc \otimes \wedge^2 E ) \ne 0$ and $P$ is not very stable. But then $P$ is not general in moduli, by \cite[Corollary 5.6]{BR}. Then the smoothness statement follows from Lemma \ref{Zariski}. \end{proof}

\section{Orthogonal extensions and principal parts} \label{OrthExt}

Let $F$ be a vector bundle of rank $n$. An extension of vector bundles $0 \to F \to V \to F^* \otimes L \to 0$ will be called an \textsl{orthogonal extension} if $V$ carries an $L$-valued quadratic form with respect to which $F$ is isotropic. Similarly to the symplectic case \cite{CCH1}, an orthogonal extension is a key ingredient in the proof of Theorems \ref{Even} and \ref{Odd}. Here we recall or prove various results on isotropic subbundles of such extensions. 

Many of the proofs in this section are virtually identical to those in \cite{CCH1} of the corresponding statements for symplectic extensions, so we shall sometimes omit details and refer to the corresponding arguments in \cite{CCH1}.

\subsection{Principal parts} Recall that any locally free sheaf $W$ on $C$ has a flasque resolution
\[ 0 \ \to \ W \ \to \ \sRat (W) \ \to \ \sPrin (W) \ \to \ 0 , \]
where $\sRat (W) = W \otimes_{\Oc} \sRat (\Oc)$ is the sheaf of sections of $W$ with finitely many poles, and $\sPrin (W) = \sRat (W)/W$ is the sheaf of principal parts with values in $W$. Taking global sections, we have a sequence of Abelian groups
\begin{equation} 0 \ \to \ H^0 (C, W) \ \to \ \Rat (W) \ \to \ \Prin (W) \ \to \ H^1 (C, W) \ \to \ 0 . \label{cohomseq} \end{equation}
Unpacking the definition of quotient sheaf, a principal part $p \in \Prin (W)$ is determined by an open covering $\{ U_i : 1 \le i \le t \}$ of $C$ together with a rational section $w_i \in \sRat(W) (U_i)$ for each $i$; another choice of $\{ w_i' \}$ determines the same element of $\Prin (W)$ if and only if $w_i - w_i'$ is regular for each $i$. Clearly the cover may be assumed to be finite. Refining it if necessary, we may assume that each $w_i$ has at most one pole. Further refining the cover if necessary, we may choose a local parameter $z_i$ on $U_i$ and a section $\omega_i \in H^0 ( U_i, W)$ such that $w_i = \frac{\omega_i}{z_i^{d_i}}$ represents $p|_{U_i}$. Abusing notation, we shall often write
\[ p \ = \ \frac{\omega_1}{z_1^{d_1}} + \cdots + \frac{\omega_k}{z_t^{d_t}} . \]

\begin{definition} Let $F$ be any bundle of rank $n$. For a principal part
\[ p \ = \ \frac{\omega_1}{z_1^{d_1}} + \cdots + \frac{\omega_m}{z_m^{d_m}} \ \in \ \Prin(L^{-1} \otimes F \otimes F) , \]
the \textsl{transpose} $^tp$ is the principal part represented by $\frac{^t\omega_1}{z_1^{d_1}} + \cdots + \frac{^t\omega_k}{z_m^{d_m}}$. Then $p$ is \textsl{antisymmetric} if $^tp = -p$; equivalently, $^tp_x + p_x$ is regular for all $x$; or $p \in \Prin(L^{-1} \otimes \wedge^2 F)$. 
\end{definition}

Now any $p \in \Prin ( L^{-1} \otimes F \otimes F)$ defines naturally an $\Oc$-module map $F^* \otimes L \to \sPrin (F)$, which we also denote by $p$. Suppose $p$ is an antisymmetric principal part in $\Prin ( L^{-1} \otimes \wedge^2 F )$. Following \cite[Chapter 6]{K1}, we define a sheaf $V^p$ over $C$ by
\begin{equation} V^p (U) \ := \ \{ (f, \varphi) \in \sRat(F) (U) \oplus ( F^* \otimes L )(U) : \overline{f} = p (\varphi) \} . \label{V} \end{equation}
It is not hard to see that $V^p$ is an extension of $F^* \otimes L$ by $F$. Let
\[ \langle \, , \rangle \colon \sRat (F) \oplus \sRat(F^* \otimes L) \ \to \ \sRat (L) \]
be the natural pairing. By an easy computation (see the proof of \cite[Criterion 2.1]{Hit1} for a more general case), the $\sRat(L)$-valued quadratic form
\begin{equation} \sigma \left( ( f_1 , \varphi_1 ) , ( f_2 , \varphi_2 ) \right) \ = \ \langle f_2 , \varphi_1 \rangle + \langle f_1 , \varphi_2 \rangle \label{standardquadform} \end{equation}
on $\sRat(F) \oplus \sRat(F^* \otimes L)$ restricts to a \emph{regular} $L$-valued nondegenerate quadratic form on $V^p$ with respect to which the subsheaf $F$ is isotropic. Thus for each antisymmetric principal part $p \in  \Prin(L^{-1} \otimes \wedge^2 F)$ there is a naturally associated orthogonal extension. We now state a refinement of \cite[Criterion 2.1]{Hit1}, telling us that every orthogonal bundle with a rank $n$ isotropic subbundle can be put into this form.

\begin{lemma} \label{orthext} Let $V$ be an $L$-valued orthogonal bundle of rank $2n$ admitting a rank $n$ isotropic subbundle $F$.
\begin{enumerate}
\item[(a)] There is an isomorphism of orthogonal bundles $\iota \colon V \isom V^p$  for some antisymmetric principal part $p \in \Prin (L^{-1} \otimes \wedge^2 F)$ such that $\iota(F)$ is the natural copy of $F$ in $V^p$ given by $\{ ( f, 0 ) : f \in F \} = V^p \cap \sRat(F)$.
\item[(b)] The class of the extension $0 \to F \to V^p \to F^* \otimes L \to 0$ in $ H^1 ( C, L^{-1} \otimes \wedge^2 F)$ coincides with $[p]$.
\end{enumerate} \end{lemma}

\begin{proof} This is identical to the proof of \cite[Lemma 3.1]{CCH1} up to changes of sign. \end{proof}

\subsection{Maximal rank isotropic subbundles of a fixed extension}

From (\ref{V}), we obtain a splitting $\Rat (V) = \Rat (F) \oplus \Rat( F^* \otimes L )$. This is a vector space of dimension $\rank(V)$ over the field $K(C)$ of rational functions on $C$. If $\beta \in \Rat ( \Hom ( F^* \otimes L , F ))$, we write $\Gamma_{\beta}$ for the graph of the induced map of $K(C)$-vector spaces $\Rat (F^* \otimes L) \to \Rat (F)$. We denote by $\underline{\Gamma_\beta}$ the associated $\Oc$-submodule of $\sRat (F) \oplus \sRat (F^* \otimes L)$.

\begin{proposition} Let $p \in \Prin ( L^{-1} \otimes \wedge^2 F )$ be any antisymmetric principal part. Let $V^p$ be as in (\ref{V}). \label{cohomlifting} \begin{enumerate} 
\item[(a)] There is a bijection between the $K(C)$-vector space $\Rat \left( L^{-1} \otimes \wedge^2 F \right)$ and the set of rank $n$ isotropic subbundles $E \subset V^p$ with $\rank (E \cap F) = 0$. The bijection is given by $\beta \mapsto \underline{\Gamma_\beta} \cap V^p$. The inverse map sends a rank $n$ isotropic subbundle $E$ to the uniquely determined $\beta \in \Rat( L^{-1} \otimes \wedge^2 F)$ satisfying $\Rat(E) = \Gamma_{\beta}$.
\item[(b)] If $E = \underline{\Gamma_\beta} \cap V^p$, then projection to $F^* \otimes L$ gives an isomorphism of sheaves
\[ E \ \isom \ \Ker \left( (p - \overline{\beta}) \colon F^* \otimes L \to \sPrin (F) \right) . \]
Note that $\left[ p - \overline{\beta} \right] = [p] $ is the class of $V^p$ in $H^1(C, L^{-1} \otimes \wedge^2 F)$.
\item[(c)] For a fixed $p - \overline{\beta} \in \Prin ( L^{-1} \otimes \wedge^2 F )$, the set of rank $n$ isotropic subbundles $\underline{\Gamma_{\beta'}} \cap V^p$ with $\overline{\beta'} = \overline{\beta}$ is a torsor over $H^0 ( C, L^{-1} \otimes \wedge^2 F)$. In particular, it is nonempty.
 \end{enumerate} \end{proposition}

\begin{proof} Up to changes of sign, parts (a), (b) and (c) are proven identically to \cite[Theorem 3.3 (i) and (iii)]{Hit1} and 
 \cite[Proposition 3.2 (c)]{CCH1} respectively. \end{proof}

The following refinement of Lemma \ref{orthext} allows us to choose convenient coordinates on $V$.

\begin{lemma} Let $F$ and $E$ be rank $n$ isotropic subbundles of $V$ such that $\rank (F \cap E) =  0$. Then there exists a antisymmetric principal part $p_0 \in \Prin ( L^{-1} \otimes \wedge^2 F)$ and an isomorphism of orthogonal bundles $\iota \colon V \isom V^{p_0}$, such that
\[ \iota(E) \ = \ \underline{\Gamma_0} \cap V^{p_0} \ = \ \{0\} \oplus \Ker (p_0) , \]
where $\Gamma_0 = \{ 0 \} \oplus \Rat (F^* \otimes L)$ is the graph of the zero map $\Rat (F^* \otimes L) \to \Rat ( F )$. \label{GoodCoords} \end{lemma}

\begin{proof} This is identical to \cite[Lemma 3.4]{CCH1} up to changes of sign. \end{proof}

\section{Components of isotropic Quot schemes in even rank} \label{EvenRankComponents}

In this section, we give a first description of the irreducible components of $\IQe$. We shall say that a locus in $\IQe$ is \textsl{saturated} if its generic point represents a subbundle of $V$, and \textsl{nonsaturated} otherwise.

\subsection{Nonsaturated loci in even rank}

Let $V$ be an $L$-valued orthogonal bundle of rank $2n$. For any $E_1 \in \IQ_{e_1}^\circ (V)$ and $e < e_1$, any $[E \to E_1] \in \Elm^{e_1 - e} (E_1)$ defines a point $[E \to E_1 \to V]$ of $\IQe$. However, the next theorem and corollary show that in general such an $[E \to V]$ does not lie in the closure $\overline{\IQeo}$ of the saturated locus.

\begin{theorem} \label{ComponentsOfIQe} Let $V$ be an $L$-valued orthogonal bundle of rank $2n \ge 4$.
\begin{enumerate}
\item[(a)] Suppose that $e < e_1$ and $\IQ_{e_1}^\circ (V)$ is nonempty. For each irreducible component $Y$ of $\IQ^\circ_{e_1} (V)$, let $Y ( e_1 , e ) \subseteq \IQe$ be the sublocus consisting of rank $n$ isotropic subsheaves whose saturations define points of $Y$. Then $Y( e_1, e )$ is irreducible and $\dim Y(e_1, e) = \dim (Y) + n(e_1 - e)$. Moreover, if $Y$ contains a smooth reduced point, then $\overline{Y ( e_1, e )}$ is an irreducible component of $\IQe$ which is generically smooth.
\item[(b)] The map $Y ( e_1, e ) \to Y$ defined by $E \mapsto \bE$ is projective.
\item[(c)] If $Y$ and $Y'$ are generically smooth irreducible components of $\IQ^\circ_{e_1} (V)$, the corresponding components $Y ( e_1, e )$ and $Y' (e_1, e)$ are same if and only if $Y' = Y$.
\end{enumerate} \end{theorem}

\begin{proof} (a) Let $\cE$ be the restriction of the universal subsheaf of $\pi_C^* V \to \IQ^\circ_{e_1} (V) \times C$ to $Y \times C$. As $\cE$ is a family of sheaves parameterized by $Y$, the relative Quot scheme
\begin{equation} \cElm^{e_1 - e} ( \cE ) \ := \ \cQ uot_{n, e} ( \cE ) \ \to \ Y \label{RelativeQuot} \end{equation}
parameterizes rank $n$ isotropic subsheaves of $V$ of degree $e$ whose saturations define points of $Y$. By the universal property of Quot schemes, we obtain a morphism $\cElm^{e_1 - e} ( \cE ) \to \IQe$, which is clearly injective
 with image $Y ( e_1, e )$. Therefore,
\[ \dim \left( Y ( e_1, e ) \right) \ = \ \dim (  \cElm^{e_1 - e} ( \cE ) ) \ = \ \dim (Y) + n(e_1-e) . \]
Since the base and the fibers of $\cElm^{e_1-e} ( \cE ) \to Y$ are irreducible, so is $Y( e_1, e )$. 

Let us now show that $Y ( e_1, e )$ is generically smooth and reduced, and moreover is not contained in a component of higher dimension. All of this will follow if we can show that $\dim \left( T_E \IQe \right) = \dim \left( Y ( e_1, e ) \right)$ at a general point $E \in Y ( e_1, e )$.

Consider, then, an elementary transformation $0 \to E \to E_1 \to T \to 0$ which is general in the sense that $E_1$ is a smooth, reduced point of $\IQ_{e_1}^\circ (V)$ and $T \cong \cO_D$ for a reduced divisor $D$ of degree $e_1 - e$. As before, by dualizing we obtain $0 \to E_1^* \otimes L \to E^* \otimes L \xrightarrow{b} T' \to 0$ where $T'$ is noncanonically isomorphic to $T$. 

By Lemma \ref{Zariski}, the tangent space $T_E \IQe$ is the preimage of $H^0 ( C, \wedge^2 E^* \otimes L )$ by the map $c$ in the diagram
\[ \xymatrix{ H^0 ( C, E^* \otimes T ) \ar@{^{(}->}[r] & H^0 ( C, E^* \otimes V/E ) \ar@{->>}[r]^-a \ar[dr]_c & H^0 ( C, E^* \otimes (E_1^* \otimes L) ) \ar[d] \\
 & & H^0 ( C, E^* \otimes (E^* \otimes L) ) \ar[d]^{\Iden_{E^*} \otimes b} \\
 & & H^0 ( C, E^* \otimes T'). } \]
As $c$ factorizes via $H^0 ( C, E^* \otimes E_1^* \otimes L ) )$ and $a$ is surjective, there is an exact sequence
\[ 0 \to H^0 ( C, E^* \otimes T)) \to T_E \IQe \to \Pi \to 0 , \]
where
\begin{multline*} \Pi \ = \ H^0 ( C, E^* \otimes E_1^* \otimes L ) ) \cap H^0 ( C, \wedge^2 E^* \otimes L ) \\
 = \ \Ker \left( H^0 ( C, \wedge^2 E^* \otimes L ) \xrightarrow{\Iden_{E^*} \otimes b} H^0 ( C, E^* \otimes T' ) \right) . \end{multline*}
 
Let us describe the kernel of the sheaf map
\[
\Iden_{E^*} \otimes b \colon \wedge^2 E^* \otimes L \ \to \ E^* \otimes T' .
\]
At a point $x \in \Supp (T)$, let $\varphi_1 , \ldots , \varphi_n$ be a local basis for $E^*$ such that the subsheaf $E_1^*$ is given by $\varphi_1, \ldots , \varphi_{n-1}, z\varphi_n$ for a local parameter $z$ at $x$. Let $l$ be a local generator for $L$. Then $T'_x$ is generated by $b(\varphi_n \otimes l)$, and $\Ker \left( b_x \colon E^* \otimes L |_x \to T'_x \right)$ is spanned by $\varphi_1 \otimes l , \ldots , \varphi_{n-1} \otimes l$. For $1 \le i \le n-1$, we have
\[ (\Iden_{E^*} \otimes b) ( (\varphi_i \wedge \varphi_n) \otimes l ) \ = \ \varphi_i \otimes b(\varphi_n \otimes l) - \varphi_n \otimes b(\varphi_i \otimes l ) \ = \ \varphi_i \otimes b (\varphi_n \otimes l ) . \]
Hence $(\Iden_{E^*} \otimes b)( \wedge^2 E^* \otimes L )$ contains at least $n-1$ independent elements of $E^*|_x \otimes T'_x$ at each of the $e_1-e$ points $x \in \Supp (T')$. Therefore,
\[
 \deg \left( \Ker \left(\Iden_{E^*} \otimes b \right) \right) \ \le \ \deg (\wedge^2 E^* \otimes L ) - (n-1)(e-e_1) \ = \ \deg ( \wedge^2 E_1^* \otimes L ) . 
 \]
Since clearly $\wedge^2 E_1^* \otimes L$ is contained in $\Ker \left(  \Iden_{E^*} \otimes b \right)$, they must be equal. 
 Therefore 
\[ \dim ( \Pi ) \ = \ h^0 (C, \wedge^2 E_1^* \otimes L ) \ = \ \dim ( T_{E_1} \IQ_{e_1}^\circ ( V ) ) . \]
As we have supposed $Y$ to be smooth at $E_1$, this coincides with $\dim (Y)$. Thus, as desired,
\[ \dim T_E \IQe \ = \ \dim (Y) + n(e_1 - e) \ = \ \dim ( Y ( e_1, e ) ) . \]


(b) This follows from the projectivity of $\cQuot_{n, e} ( \cE )$ over $Y$ (see e.g.\ \cite[Theorem 4.4.1]{Sern06}).

(c) If $Y(e_1, e) = Y'(e_1, e)$, then we get $Y = Y'$ by the saturation process. 
\end{proof}

\begin{corollary} \label{GeneralPathology} Let $V$ be an orthogonal bundle of rank $2n$. Then for $e_1 > e$, all nonempty components of $\IQe$ of the form $Y ( e_1, e )$ have dimension strictly larger than $I ( n, \ell, e)$. \end{corollary}

\begin{proof} For $e_1 > e$, each nonempty component $Y$ of $\IQ_{e_1}^\circ (V)$ has dimension at least $I(n, \ell, e_1)$. By Theorem \ref{ComponentsOfIQe} (a), therefore,
\[ \dim Y(e_1, e) \ \ge \ I(n, \ell, e_1) + n(e_1 - e) \ = \ I(n,\ell, e) + e_1 - e . \]
As $e_1 > e$, this is strictly greater than $I(n, \ell, e)$. \end{proof}

\subsection{The closure of the saturated locus}

Now we prove Theorem \ref{TypeT} (a), which gives a necessary condition for $[E \to V] \in \IQe$ to be a limit of points in $\IQeo$. According to Theorem \ref{TypeT} (b), it is also a sufficient condition for $e$ sufficiently small, but the proof of the latter will be postponed to \S\:6. 

Recall the notion of torsion sheaves \textsl{of type} $\cT$ given in Definition \ref{TorsionTypeT}. Furthermore, if $F$ is any locally free sheaf, for $t \ge 1$, we denote by $\cT^{2t} ( F )$ the sublocus of $\Elm^{2t} (F)$ of those $[ E \to F ]$ such that $F/E$ is of type $\cT$.

\begin{lemma} \label{TypeTParameterSpace} Let $F$ be a bundle of rank $n \ge 2$.
\begin{enumerate}
\item[(a)] The locus $\cT^{2t} (F)$ is closed and irreducible in $\Elm^{2t} ( F )$.
\item[(b)] A general $E \in \cT^{2t} (F)$ satisfies $F/E \cong \cO_D \otimes \cc^2$ for some reduced effective divisor $D$, which may be taken to be general in $C^{(t)}$.
\item[(c)] The dimension of $\cT^{2t} ( F )$ is $t ( 2n - 3 )$.
\end{enumerate} \end{lemma}

\begin{proof} (a) Let $0 \to E \to F \to \tau \to 0$ be an elementary transformation as above. By hypothesis, there is a filtration
\[ 0 \ = \ \tau_0 \ \subset \ \tau_1 \ \subset \ \cdots \ \subset \ \tau_{t-1} \ \subset \ \tau_t \ = \ \tau \]
where $\tau_i / \tau_{i-1} = \cO_{x_i} \otimes \cc^2$ for $1 \le i \le t$. Denoting by $E_i$ the inverse image  of $\tau_i$ in $F$, we obtain a chain of subsheaves
\begin{equation} E \ = \ E_0 \ \to \ E_1 \ \to \ \cdots \ \to \ E_{t-1} \ \to \ E_t \ = \ F , \label{ChainOfETs} \end{equation}
where $E_i / E_{i-1} = \cO_{x_i} \otimes \cc^2$ for each $i$. Thus it will suffice to construct a space parameterizing chains of elementary transformations of the form (\ref{ChainOfETs}) which is projective and irreducible. The construction of \cite[Lemma 4.2]{CH1} can be modified in a natural way to yield such a space as a tower of Grassmann bundles, substituting $\Gr (2, \cF_k^* )$ for $\pp \cF_k^*$ and $\det ( \cU )$ for $\cO_{\pp \cF_k^*} (-1)$, where $\cU \to \Gr (2, \cF_k^*)$ is the universal bundle. As this is technical but straightforward, we omit the details. 

(b) The association $E \mapsto \frac{1}{2} \Supp ( F/E )$ defines a map $\cT^{2t} (F) \to C^{(t)}$, which is clearly surjective. Hence
\begin{equation} \{ E \ \in \ \cT^{2t} ( F ) : F/E \ \cong \ \cO_D \otimes \cc^2 \hbox{ for some reduced $D$ of degree $t$} \} \label{reducedD} \end{equation}
is a nonempty open subset of $\cT^{2t} ( F )$, being the complement of the inverse image of the big diagonal of $C^{(t)}$ by the above map. By the irreducibility proven in part (a), it is dense. Moreover, if $E$ is sufficiently general in $\cT^{2t} ( F )$ then $D$ may be taken to be general.

(c) The set (\ref{reducedD}) is in canonical bijection with the set of $t$-tuples of points of the Grassmann bundle $\Gr (2, F)$ lying over distinct points of $C$. As we saw in part (a) that $\cT^{2t}(F)$ is irreducible, $\dim \cT^{2t} (F) = t \cdot \dim \Gr (2, F ) = t ( 2n-3)$. \end{proof}

\noindent Suppose now that $V$ is $L$-valued orthogonal of rank $2n$, and that $E$ and $F$ are rank $n$ isotropic subbundles intersecting generically in dimension zero. Then by Lemma \ref{GH}, the dimension of $E|_x \cap F|_x$ is even for all $x \in C$. The following lemma slightly generalizes this fact, and shows the relevance of the type $\cT$ property for orthogonal extensions.

\begin{lemma} \label{lemmaTypeT} Let $V$ be an $L$-valued orthogonal bundle of rank $2n \ge 4$. Suppose $E$ and $F$ are rank $n$ isotropic subbundles intersecting generically in rank zero, so that the composition $E \to V \xrightarrow{q} V/F \cong F^* \otimes L$ is an elementary transformation. Then $(F^* \otimes L) / E$ is of type $\cT$. \end{lemma}

\begin{proof} By Proposition \ref{cohomlifting} (a) and (b), there exists $p \in \Prin ( L^{-1} \otimes \wedge^2 F )$ such that $E$ is isomorphic to the locally free subsheaf of $F^* \otimes L$ given by
\[ \Ker \left( p \colon F^* \otimes L \to \sPrin(F) \right) \ = \ \{ \phi \in F^* \otimes L : p(\phi) \hbox{ is regular} \} . \]
Thus $(F^* \otimes L) / E \cong \Image \left( p \colon F^* \otimes L \to \sPrin(F) \right)$. Let us describe $\Image (p)$.

By \cite[Lemma 2.6]{CH3}, for each $x \in \Supp(p)$ there exists a frame $f_1 , \ldots , f_n$ for $F$ near $x$ and a local parameter $z$ at $x$ such that 
\[ p_x \ = \sum_{i=1}^s \frac{l^* \otimes f_{2i-1} \wedge f_{2i}}{z^{d_i}} \]
for positive integers $d_1, \ldots , d_s$ with $1 \le s \le n/2$, and $l^*$ a local generator for $L^{-1}$. From this we see that $\Image(p_x) = \bigoplus_{i=1}^s T_i$, where $T_i$ is the torsion sheaf generated over $\Oc$ by $\frac{f_{2i-1}}{z^{d_i}}$ and $\frac{f_{2i}}{z^{d_i}}$. For $1 \le i \le s$, we have a filtration
\[ 0 \ = \ z^{d_i} \cdot T_i \ \subset \ z^{d_i - 1} \cdot T_i \ \subset \ \cdots \ \subset z \cdot T_i \ \subset \ T_i . \]
For $0 \le j \le d_i - 1$, the quotient $\left( z^j \cdot T_i \right) / \left( z^{j+1} \cdot T_i \right)$ is generated by the images of $z^j f_{2i-1}$ and $z^j f_{2i}$, so it is of the form $\cO_x \otimes \cc^2$. Hence each $T_i$ is of type $\cT$, and so $\Image (p) = (F^* \otimes L) / E$ is of type $\cT$. \end{proof}

\begin{remark} It follows from the proof of Lemma \ref{lemmaTypeT} that for any $p \in \Prin( L^{-1} \otimes \wedge^2 F^*)$, the elementary transformation $\Ker (p) \subseteq F^* \otimes L$ is of type $\cT$. The converse does not hold, however: Consider the exact sequence
\[ 0 \to \frac{\Oc(x)}{\Oc} \oplus \frac{\Oc(x)}{\Oc} \ \to \ \frac{\Oc(x)}{\Oc} \oplus \frac{\Oc(x)}{\Oc} \oplus \frac{\Oc(2x)}{\Oc} \ \xrightarrow{\alpha} \ \frac{\Oc(x)}{\Oc} \oplus \frac{\Oc(x)}{\Oc} \ \to \ 0 \]
where the quotient map $\alpha$ is given by
\[ \alpha \left( \frac{a}{z} , \frac{b}{z} , \frac{c+dz}{z^2} \right) \ = \ \left ( \frac{a}{z} , \frac{c}{z} \right) . \]
Thus the middle term is a torsion sheaf of type $\cT$ which is not of the form $\Image(p)$ for any $p \in \Prin ( L^{-1} \otimes \wedge^2 F )$.
\end{remark}

The following is a restatement of Theorem \ref{TypeT} (a).

\begin{proposition} \label{Smoothable} Let $\cE \to Z \times C$ be a family of rank $n$ isotropic subsheaves of $V$ parameterized by an irreducible variety $Z$. Suppose that $\cE_z$ is saturated in $V$ for generic $z \in Z$. Then for all $z \in Z$, the quotient $\bcE_z / \cE_z$ is of type $\cT$. \end{proposition}

\begin{proof} Fix $z_0 \in Z$. If $\cE_{z_0}$ is saturated, then there is nothing to prove. Suppose $\cE_{z_0} \ne \overline{\cE_{z_0}}$. To ease notation, write $\cE_{z_0} =: E$.

Choose $p \in C \setminus \Supp \left( \bE / E \right)$, and set $U := C \setminus \{ p \}$. As $p$ is an ample divisor, $U$ is affine, so $H^1 ( U, L^{-1} \otimes \bE^{\otimes 2} ) = 0$ by \cite[Theorem III.3.7]{Ha}. Hence
\[ V|_U \ \cong \ \bE|_U \oplus \ (\bE^* \otimes L)|_U . \]
Let $B := x_1 + \cdots + x_k$ be the reduced divisor underlying $\Supp ( \bE / E )$. For $1 \le i \le k$, choose an $n$-dimensional isotropic subspace $\Lambda_i \subset V|_{x_i}$ intersecting $\bE|_{x_i}$ in zero. (Notice that by Lemma \ref{GH}, all such $\Lambda_i$ belong to the same component of $\OG (V)$.) By linear algebra, such a $\Lambda_i$ is the graph of a unique antisymmetric map $\omega_i \colon \bE^* \otimes L|_{x_i} \to \bE|_{x_i}$ for all $i$. Since $U$ is affine, we may assume $h^1 ( U, L^{-1}(-B) \otimes \wedge^2 \bE ) = 0$. Thus there exists $\omega \in H^0 ( U, L^{-1} \otimes \wedge^2 \bE )$ such that $\omega|_{x_i} = \omega_i$. The graph of $\omega$ is then a rank $n$ isotropic subbundle $F_U$ of $V|_U$ satisfying $F|_x \cap \bE|_x = 0$ for each $x \in \Supp ( \bE / E )$. As $C$ is a curve, $F_U$ extends uniquely to a rank $n$ isotropic subbundle of $V$. Then $V$ is an orthogonal extension $0 \to F \to V \xrightarrow{q} F^* \otimes L \to 0$.

Shrinking $Z$ if necessary, we may assume that $\rank ( \cE_z \cap F ) = 0$ for all $z \in Z$. By Lemma \ref{lemmaTypeT}, the elementary transformation $q \colon \cE_z \to F^* \otimes L$ has a torsion quotient of type $\cT$ for all $z$ in the open subset of $Z$ such that $\cE_z$ is saturated. As by Lemma \ref{TypeTParameterSpace} (a) the locus $\cT^{2k} ( F^* \otimes L )$ is closed, it also contains the limit $[q \colon E \to F^* \otimes L]$, so $(F^* \otimes L) / q ( E )$ is also of type $\cT$.

On the other hand, since $\bE$ is an isotropic subbundle, $(F^* \otimes L) / q (\bE)$ is also of type $\cT$ by Lemma \ref{lemmaTypeT}. Now there is an exact sequence
\[ 0 \ \to \ \bE / E \ \to \ (F^* \otimes L) /q(E) \ \to \ (F^* \otimes L) / q(\bE) \ \to \ 0 , \]
where both $F^* \otimes L / q(E)$ and $F^*/q(\bE)$ are of type $\cT$. By the construction of $F$ above, $\Supp \left( (F^* \otimes L)/q(\bE) \right)$ and $\Supp ( \bE / E )$ are disjoint. Now clearly a torsion sheaf $\Sigma$ is of type $\cT$ if and only if $\cT_x$ is of type $\cT$ for all $x \in \Supp ( \Sigma )$. Therefore, $\bE / E$ is of type $\cT$. \end{proof}

\begin{remark} \label{SmoothabilityCounterexample} In general, not every rank $n$ isotropic subsheaf $E \subset V$ with $\overline{E}/E$ of type $\cT$ need deform to an isotropic subbundle. For example, consider rank four $\Oc$-valued orthogonal extensions $0 \to F \to V \to F^* \to 0$ where $F$ is a general stable bundle of rank two and degree $-1$. For any $x \in C$, the elementary transformation $F(-x) \subset F$ is a rank $n$ isotropic subsheaf of degree $-3$ with $F/F(-x)$ of type $\cT$. However, if $g \ge 5$ and the chosen extension class for $V$ is general, then by an argument similar to that in the proof of \cite[Proposition 3.3]{CH3}, there are no rank two isotropic subbundles of degree $-3$ in $V$.

However, we shall prove in {\S}\:\ref{EvenRank} that for sufficiently small $e$, every rank $n$ isotropic subsheaf $E$ of degree $e$ with $\bE / E$ of type $\cT$ does deform to an isotropic subbundle. \end{remark}

\section{Orthogonal extensions and isotropic liftings} \label{OrthExtII}

In this section, we study liftings in orthogonal extensions more closely. Some proofs are essentially identical to the symplectic case treated in \cite[{\S} 3]{CCH1}, but we include some details for the reader's convenience.

Let $V$ be an $L$-valued orthogonal bundle isomorphic to an orthogonal extension $0 \to F \to V^p \to F^* \otimes L \to 0$ as defined in (\ref{V}), and let $0 \to E \xrightarrow{\gamma} F^* \otimes L \to \tau \to 0$ be an elementary transformation where $\tau$ is a torsion sheaf. Assume that there is a lifting $j \colon E \to V$. By Proposition \ref{cohomlifting}, there exists a rational map $\beta \colon \sRat (F^* \otimes L) \to \sRat(F)$ such that $E \subseteq \Gamma_\beta \cap V^p \cong \Ker ( p - \overline{\beta} )$. The  following result, generalizing Proposition \ref{cohomlifting} (c), provides the main idea to ``linearize''  the collection of rank $n$ isotropic subsheaves of $V$ intersecting $F$ generically in zero.

\begin{lemma} \label{DifferentLiftings} The set of liftings of $\gamma \colon E \to F^* \otimes L$ to isotropic subsheaves of $V = V^p$ is a torsor over $H^0 \left( C, \Hom ( E, F ) \cap \sRat ( L^{-1} \otimes \wedge^2 F ) \right)$. \end{lemma}

\begin{proof} This is identical to \cite[Lemma 3.8]{CCH1} up to changes of sign. Let us just indicate how the intersection of $\Hom (E, F)$ and $\sRat ( L^{-1} \otimes \wedge^2 F )$ is well defined. Since $L^{-1} \otimes F \xrightarrow{^t\gamma} E^*$ is an elementary transformation, $E^*$ is a subsheaf of $\sRat (L^{-1} \otimes F)$. Hence $\Hom (E, F) = E^* \otimes F$ and $\sRat ( L^{-1} \otimes \wedge^2 F)$ are both $\Oc$-submodules of $\sRat ( L^{-1} \otimes F \otimes F)$. \end{proof}

\noindent Motivated by Lemma \ref{DifferentLiftings}, we define a sheaf $A_\gamma$ as follows.

\begin{definition} \label{defAgamma} Let $0 \to E \xrightarrow{\gamma} F^* \otimes L \to \tau \to 0$ be an elementary transformation defining a point of $\Elm^k ( F^* \otimes L )$ for some $k \ge 0$. From the transpose $F \otimes L^{-1} \xrightarrow{^t\gamma} E^*$ we deduce an elementary transformation $L^{-1} \otimes F \otimes F \to E^* \otimes F$. We define $A_\gamma$ to be the saturation of $L^{-1} \otimes \wedge^2 F$ in $E^* \otimes F$. Note that
\[ A_\gamma \ = \ \Hom ( E, F ) \cap \sRat ( A_\gamma ) \ = \ \Hom (E, F) \cap \sRat ( L^{-1} \otimes \wedge^2 F ) . \]
\end{definition}

\noindent We remark also that the definition of $A_\gamma$ depends only on $\gamma$, and does not make reference to an extension $0 \to F \to V \to F^* \otimes L \to 0$.

\begin{lemma} \label{Agamma} Let $\gamma \colon E \to F^* \otimes L$ and $A_\gamma$ be as above.
\begin{enumerate}
\item[(a)] There is a short exact sequence
\begin{equation} 0 \ \to \ L^{-1} \otimes \wedge^2 F \ \to \ A_\gamma \ \to \ \tau_1 \ \to \ 0, \label{Agammaparta} \end{equation}
where $\tau_1$ is a torsion sheaf. In particular, $A_\gamma$ is locally free of rank $\frac{1}{2}n(n-1)$.
\item[(b)] There is a short exact sequence
\[ 0 \ \to \ A_\gamma \ \to \ L \otimes \wedge^2 E^* \ \to \ \tau_2 \ \to \ 0, \]
where $\tau_2$ is a torsion sheaf.
\item[(c)] Suppose $\frac{F^* \otimes L}{E} \cong \cO_D \otimes \cc^2$ for a reduced $D \in C^{(t)}$. Then $\tau_1 \cong \cO_D$. In this case, $\deg (A_\gamma) = \deg (L^{-1} \otimes \wedge^2 F) + t$. \end{enumerate}
\end{lemma}

\begin{proof} Part (a) follows from the definition, and part (b) is identical to \cite[Lemma 3.8 (b)]{CCH1}, substituting $\wedge^2$ for $\Sym^2$.
 As for (c): Since $(F^* \otimes L)/E \cong \cO_D \otimes \cc^2$, also $E^* / (L^{-1} \otimes F) \cong \cO_D \otimes \cc^2$. At each $x \in \Supp (\tau) = \Supp(D)$, a local basis for $E^* \subset \sRat (L^{-1} \otimes F)$ is
\[ \frac{l^* \otimes v_1}{z} , \frac{l^* \otimes v_2}{z} , l^* \otimes v_3 , \ldots, l^* \otimes v_n , \]
where $\{ v_1, \ldots , v_n \}$ is a suitable local basis of $F$ and $l^*$ is a local generator of $L^{-1}$, and $z$ is a local parameter at $x$. Then a local basis of $E^* \otimes F$ is given by
\[
\left\{ \frac{l^* \otimes v_i \otimes v_j}{z} : \begin{array}{c} 1 \le i \le 2 ; \\ 1 \le j \le n \end{array} \right\} \: \cup \: \left\{l^* \otimes v_i \otimes v_j : \begin{array}{c} 3 \le i \le n ; \\ 1 \le j \le n \end{array} \right\}.
\]
Thus a local basis of $A_\gamma$ is
\[ \left\{ \frac{l^* \otimes (v_1 \wedge v_2)}{z} \right\} \ \cup \ \left\{ l^* \otimes (v_i \wedge v_j) : \begin{array}{c} 1 \le i < j \le n ; \\ (i, j) \ne (1, 2) \end{array} \right\} . \]
Thus the torsion sheaf $A_\gamma / \left( L^{-1} \otimes \wedge^2 F \right)$ is supported along $D$ and has length 1 at each point of $D$. Part (c) follows. \end{proof}

\begin{remark} \label{nonreduced} In analogy with \cite[Remark 3.8]{CCH1}, we observe that part (c) may be false if $D$ is not reduced. Suppose for example that $L = \Oc$ and that $E^* / F \cong \frac{\Oc(x)}{\Oc} \otimes \Lambda$ for some subspace $\Lambda \subseteq F$. Then there is an exact sequence $0 \to \wedge^2 F \to A_\gamma \to \frac{\Oc(x)}{\Oc} \otimes \wedge^2 \Lambda \to 0$. Thus $\tau_1 \cong \cO_x$ if and only if $\Lambda$ has dimension $2$. \end{remark}

Next, we discuss some geometry in orthogonal extension spaces which we shall require later. Let $F \to C$ be a bundle of rank $n$, and consider the Grassmann bundle $\pi \colon \Gr(2, F) \to C$. For the remainder of this section, we shall assume $h^1 ( C, L^{-1} \otimes \wedge^2 F ) \ne 0$. By Serre duality and the projection formula, there is an isomorphism
\[ \pp H^1 ( C, L^{-1} \otimes \wedge^2 F) \ \xrightarrow{\sim} | \cO_{\pp (\wedge^2 F)} (1) \otimes \pi^* (\Kc L) |^* . \]
Composing with the relative Pl\"ucker embedding, we obtain a natural map
\[ \psi \colon \Gr (2, F) \ \dashrightarrow \ \pp H^1 (C, L^{-1} \otimes \wedge^2 F ) \]
with nondegenerate image. This was studied in \cite[{\S} 2.2]{CH3} and \cite[{\S} 3]{CH5}.

For each $x \in C$, there is a cohomology sequence
\begin{multline} 0 \ \to \ H^0 ( C, L^{-1} \otimes \wedge^2 F ) \ \to \ H^0 ( C, L^{-1}(x) \otimes \wedge^2 F ) \\
 \to \ \left( L^{-1}(x) \otimes \wedge^2 F \right)_x \ \to \ H^1 ( C, L^{-1} \otimes \wedge^2 F ) \label{fiberwise} \end{multline}
Now we recall an explicit description of $\psi$ given in \cite[{\S} 2.2]{CH3}. (The statement is for $L = \Oc$, but the argument for arbitrary $L$ is identical.)

\begin{lemma} The map $\psi$ can be identified fiberwise with the projectivization of the coboundary map in (\ref{fiberwise}), restricted to the image of the Pl\"ucker map $\Gr(2, F|_x) \hookrightarrow \pp ( L^{-1} \otimes \wedge^2 F )|_x$. In particular, the image of $\Lambda \in \Gr(2, F|_x)$ is the projectivized cohomology class of a principal part of the form $\frac{l^* \otimes v_1 \wedge v_2}{z}$, where $v_1, v_2$ are linearly independent elements of $\Lambda$ and $l^*$ is a local generator of $L^{-1}$, and $z$ is a local parameter at $x$. \label{alternativePsi} \end{lemma}
	
Although we do not use this fact, we mention that if $F$ is stable and $\mu ( F ) < \frac{\ell}{2} - 1$, then $\psi$ is an embedding (see \cite[Lemma 2.2 (1)]{CH3} for the case where $L = \Oc$).

Now let $\sigma_1 , \ldots , \sigma_t$ be points of $\Gr (2, F )$ lying over distinct points $x_1 , \ldots , x_t$ of $C$ respectively. Let $[\gamma \colon E \to F^* \otimes L]$ be the point of $\cT^{2t} ( F^* \otimes L )$ determined by $\sigma_1 , \ldots , \sigma_t$.

\begin{lemma} \label{vanishinglemma} We have $h^1 ( C, A_\gamma ) = 0$ if and only if the points $\psi (\sigma_1) , \ldots , \psi (\sigma_t)$ span $\pp H^1(C, L^{-1} \otimes \wedge^2 F)$.
\end{lemma}
\begin{proof}
From the proof of Lemma \ref{Agamma} (c) we see that for $1 \le i \le t$, there exists a local parameter $z_i$ at $x_i$ together with local sections $\eta_i^{(1)}$ and $\eta_i^{(2)}$ of $F$ and a local generator $l_i^*$ for $L^{-1}$ such that $A_\gamma$ is the elementary transformation
\[
0 \longrightarrow L^{-1} \otimes \wedge^2 F \longrightarrow A_\gamma \longrightarrow \bigoplus_{i=1}^t \Oc \cdot \frac{l_i \otimes \eta_i^{(1)} \wedge \eta_i^{(2)}}{z_i} \longrightarrow 0 .
\]
In view of Lemma \ref{alternativePsi}, the lemma follows from the associated long exact sequence
\begin{multline}
0 \ \longrightarrow \ H^0 ( C, L^{-1} \otimes \wedge^2 F ) \ \longrightarrow \ H^0 ( C, A_\gamma ) \ \longrightarrow \ \bigoplus_{i=1}^t \cc \cdot \frac{l_i \otimes \eta_i^{(1)} \wedge \eta_i^{(2)}}{z_i} \\ 
\longrightarrow \ H^1 ( C, L^{-1} \otimes \wedge^2 F) \ \longrightarrow \ H^1( C, A_\gamma) \ \longrightarrow \ 0 . \qedhere \label{syzygy}
\end{multline}
\end{proof}

\begin{remark} \label{syzygies} Suppose $H^0 (C, L^{-1} \otimes \wedge^2 F) = 0$. Then by exactness of (\ref{syzygy}), we see that $H^0 (C, A_\gamma)$ is the module of linear relations among the points $\psi ( \sigma_i )$ in $\pp H^1 (C, L^{-1} \otimes \wedge^2 F )$. \end{remark}

\section{The saturated components in the even rank case} \label{EvenRank}

Throughout this section, $V$ will be an $L$-valued orthogonal bundle of rank $2n$ and determinant $L^n$. In this section we shall prove the irreducibility result Theorem \ref{Even} for $V$. The approach is essentially the same as in \cite[{\S} 4]{CCH1}, but the proof is more delicate due to the nonsaturated components described in the previous section.

\subsection{Surjectivity of evaluation maps} For each $x \in C$, recall that the \textsl{evaluation map} $\ev_x \colon \IQe \dashrightarrow \OG ( V|_x )$ is the rational map taking an isotropic subsheaf to the image of the vector bundle map $E|_x \to V|_x$. The map $\ev_x$ is defined at $E \in \IQe$ if and only if $E$ is saturated at $x$.

\begin{lemma} \label{evdom} For $\delta \in \{ 1 , 2 \}$, there exists an integer $f_\delta$ such that $\ev_x \colon \IQo_{f_\delta}(V)_\delta \to \OG(V)_\delta|_x$ is surjective for general $x \in C$. \end{lemma}

\begin{proof} Since we are assuming $\det(V) = L^n$, by Lemma \ref{EvenRankExistence} there is a Zariski open subset $U \subseteq C$ over which we have an orthonormal frame for $V$. Then also $\OG(V)|_U \cong U \times \OG(n, 2n)$, and we obtain a labeling of the components $\OG(n, 2n)_\delta$ induced from $\OG(V)_\delta$.

Now each $\Lambda \in \OG(n, 2n)$ yields a rank $n$ isotropic subbundle of $V|_U$. Since $C$ is a curve, this extends uniquely to a rank $n$ subbundle $F_\Lambda$ of $V$. For each $\delta$, set
\[ f_\delta \ := \ \min \left\{ \deg \left( F_\Lambda \right) : \Lambda \in \OG( n, 2n )_\delta \right\} . \]
By semicontinuity (using the fact that the universal sheaf over any Quot scheme is flat), $\deg(F_\Lambda) = f_\delta$ for $\Lambda$ belonging to a dense open subset of the component $\OG(n, 2n)_\delta$. Then $\ev \colon \IQ_{f_\delta}^\circ (V) \to \OG ( V )_\delta|_x$ is surjective for $x \in U$.
\end{proof}

The following notation will be convenient for giving a unified treatment of all combinations of parities of $\ell$ and $n$.

\begin{definition} \label{deltapr} For $\delta \in \{ 1 , 2 \}$ and $E \in \IQe_\delta$, let $\dpr \in \{ 1, 2 \}$ be such that the component $\OG ( V )_{\dpr}$ contains those subspaces intersecting a generic fiber of $\bE$ in dimension zero. By Lemma \ref{GH}, we have $\dpr \equiv \delta + n \mod 2$. Note also that $(\delta')' = \delta$. \end{definition}

\subsection{Open cells in isotropic Quot schemes}

Following the treatment of the symplectic case in \cite[{\S} 3]{CCH1}, we shall now associate to any rank $n$ isotropic subbundle of $V$ a cell in $\IQe$. 

\begin{definition} For any rank $n$ isotropic subbundle $F \subset V$, we define
\[ \QF \ := \ \{ E \in \IQe : \rank ( E \cap F ) = 0 \hbox{ and } \bE / E \hbox{ is of type } \cT \} . \]
Note that $\QF$ also depends on the particular inclusion $F \hookrightarrow V$, but it will not be necessary to specify this in the notation. If $\QF$ is nonempty then, by Definition \ref{deltapr}, we have the key property that \emph{$\QF \subseteq \IQe_\delta$ if and only if the subbundle $F$ belongs to $\OG(V)_\dpr$}.
\end{definition}

\begin{lemma} \label{QFnonempty} Suppose $e \le f_\delta$ and $e \equiv f_\delta \mod 2$. Then $\QF$ is nonempty. \end{lemma}

\begin{proof} If $e = f_\delta$ then by Lemma \ref{evdom} we can find an isotropic subbundle $E_1 \subset V$ of degree $f_\delta$ such that $E_1|_x \cap F|_x = 0$ for some, and hence for generic $x \in C$, so $E_1 \in Q_F^{f_\delta}$. If $e < f_\delta$ then, as $e \equiv f_\delta \mod 2$, the locus $\cT^{f_\delta - e} ( E_1 )$ is nonempty and injects naturally into $\QF$. \end{proof}

We now introduce an important tool for studying $\QF$. Given $[j \colon E \to V] \in \QF$, by composing with $\pi \colon V \to F^* \otimes L$ we obtain an elementary transformation $\pi \circ j \colon E \to F^* \otimes L$. This defines a morphism $\pi_* \colon \QF \to \Elm^{-e - f + n\ell} (F^* \otimes L)$, where $f := \deg (F)$. We write $\wj$ for the map $\pi \circ j$.

\begin{lemma} \label{FirstPptiesQF} Let $F \subset V$ be as above, and let $\delta \in \{ 1, 2 \}$ be such that $F$ belongs to $\OG (V)_\dpr$.
\begin{enumerate}
\item[(a)] For any $E \in \QF$, the elementary transformation $\wj \colon E \to F^* \otimes L$ is of type $\cT$. In particular, $\QF$ is nonempty only if $-e -f + n\ell \equiv 0 \mod 2$.
\item[(b)] 
If $X$ is a component of $\IQe_\delta$ whose generic element is a saturated subsheaf, then $\QF$ is open in $X$.
\end{enumerate}
\end{lemma}

\begin{proof} (a) Firstly, $\bE / E$ is of type $\cT$ by definition of $\QF$. Write $\overline{j}$ for the natural extension of $j$ to $\bE \to V$. Then $\frac{F^* \otimes L}{\left( \widetilde{\overline{j}} \right) (\bE)}$ is of type $\cT$ by Proposition \ref{Smoothable}. Now it is easy to check that an extension of torsion sheaves of type $\cT$ is of type $\cT$. Hence the first statement follows from the short exact sequence
\[ 0 \ \to \ \frac{\left( \bE \right)}{E} \ \to \ \frac{F^* \otimes L}{\wj (E)} \ \to \ \frac{F^* \otimes L}{\widetilde{\left( \overline{j} \right)}\left( \bE \right)} \ \to \ 0 . \]
For the rest; as a torsion sheaf of type $\cT$ has even degree,
\[ \deg ( F^* \otimes L ) - \deg (E) \ = \ -e - f + n\ell \ \equiv \ 0 \mod 2 . \]

(b) The set $\{ [ E \to V ] \in X : \rank ( F \cap E ) = 0 \}$ is open in $X$, being precisely the domain of definition of $\pi_*|_X$. As a general point of $X$ is saturated, by Proposition \ref{Smoothable} the condition that $\bE/E$ be of type $\cT$ is satisfied at all points of $X$, so $\QF$ is open in $X$. \end{proof}

Note that $\bE / E$ being of type $\cT$ is only a necessary condition for $E$ to belong to $\bIQeo$ (cf.\ Remark \ref{SmoothabilityCounterexample}). Thus, a priori $\QF$ may not be contained in $\bIQeo$. However, we shall see below in Proposition \ref{QFoDense} that for $e \ll 0$, the type $\cT$ condition is also sufficient. We will use this fact in the proof of Theorem \ref{Even}, and also in Corollary \ref{compactification} to characterize the points in $\bIQeo \backslash \IQeo$.\\
\par
In view of Lemma \ref{FirstPptiesQF} (a), \textbf{we shall henceforth assume that $-e-f+n\ell \equiv 0 \mod 2$}, and write $-e-f+n\ell = 2t$.\\
\par

Next, we define a subset of $\QF \subset \IQe$ with even more desirable properties.

\begin{definition} \label{defnQFo} Let $\QFo$ be the subset of $\QF$ of those $[j \colon E \to V]$ satisfying the following:
\begin{enumerate}
\item[(i)] $E$ is saturated; that is, $j$ is a vector bundle injection.
\item[(ii)] $(F^* \otimes L) / \left( \wj(E) \right) \cong \cO_D \otimes \cc^2$ for a reduced divisor $D$ of degree $t$.
\item[(iii)] $h^1 (C, A_{\wj} ) = 0$, where $A_{\wj}$ is as in Definition \ref{defAgamma}.
\end{enumerate}
\end{definition}

\begin{remark} \label{PropertiesQFe0} \quad 
\begin{enumerate}
\item[(a)] Note that conditions (ii) and (iii) depend only on the map $\wj \colon E \to F^* \otimes L$, and not on $V$.
\item[(b)] Condition (i) is clearly open on $\QF$. By Lemma \ref{Agamma} (c), the family of sheaves over $\cT^{2t} ( F^* \otimes L )$ with fiber $A_\wj$ at $\wj$ is flat over the open subset of $\cT^{2t} ( F^* \otimes L )$ defined by condition (ii) (although, by Remark \ref{nonreduced}, not over all of $\cT^{2t} ( F^* \otimes L )$). Thus conditions (ii) and (iii) together define an open subset of $\QF$. Hence $\QFo$ is open in $\QF$.
\item[(c)] If $h^1 ( C, L^{-1} \otimes \wedge^2 F ) = 0$ then condition (iii) follows automatically by (\ref{syzygy}). Otherwise, note that by (ii), the elementary transformation $E \subset F^* \otimes L$ is determined by a choice of $\sigma_1 , \ldots , \sigma_t \in \Gr ( 2, F )$, and by Lemma \ref{vanishinglemma}, property (iii) is equivalent to the images of these points spanning $\pp H^1 ( C, L^{-1} \otimes \wedge^2 F)$.
\end{enumerate}
\end{remark}

\begin{proposition} \label{QFoSmooth} Suppose $[j \colon E \to V]$ is a point of $\QFo$.
\begin{enumerate}
\item[(a)] We have $h^1 ( C, L \otimes \wedge^2 E^* ) = 0$.
\item[(b)] Both $\QFo$ and $\IQeo$ are smooth of dimension $I(n, \ell, e)$ at $[j \colon E \to V]$.
\end{enumerate}
\end{proposition}

\begin{proof} Part (a) follows from condition (iii) and the sequence (\ref{syzygy}). Part (b) follows from part (a) by Proposition \ref{ExpDimSmoothness} (c), and noting that $\QFo$ is open in $\IQeo$ by Lemma \ref{FirstPptiesQF} (b) and Remark \ref{PropertiesQFe0} (b). \end{proof}

\begin{proposition} \label{dominant} Suppose that $\QFo$ is nonempty. For any component $X$ of $\QFo$, the map $\pi_* \colon X \to \cT^{2t}(F^* \otimes L)$ is dominant and has irreducible fibers. \end{proposition}

\begin{proof} Let $[j \colon E \to V]$ be a point of $X$, so $\left[ \wj \colon E \to F^* \otimes L \right] \in \Image (\pi_*)$. By Proposition \ref{DifferentLiftings} and Definition \ref{defAgamma}, the fiber $\pi_*^{-1} \left( \wj \right)$ is an open subset of a torsor over $H^0 ( C, A_\wj )$. In particular, it is irreducible.

Let us compute $\dim \pi_*^{-1} \left( \wj \right) = h^0 ( C, A_\wj )$. By definition of $\QFo$, we have $h^1 ( C, A_\gamma ) = 0$, so $\dim \pi_*^{-1} \left( \wj \right) = \chi ( C, A_\wj )$. Thus, using Proposition \ref{QFoSmooth} (b), we have
\begin{multline} \dim \pi_* (X) \ \ge \ \dim \QFo - \dim \pi_*^{-1} \left( \wj \right) \ = \ \chi ( C, L \otimes \wedge^2 E^* ) - \chi ( C, A_\wj ) \\
 = \ \deg( L \otimes \wedge^2 E^* ) - \deg ( A_\wj ), \label{dimpistarX} \end{multline}
the last equality because the bundles have the same rank. By condition (iii) and Proposition \ref{Agamma} (c) we have
\[ \deg ( A_\wj ) \ = \ \deg ( L^{-1} \otimes \wedge^2 F ) + t \ = \ (n-1)f - \frac{n(n-1)}{2}\cdot \ell + t , \]
so the number (\ref{dimpistarX}) is
\begin{multline*} \left( \frac{n(n-1)}{2} \cdot \ell - (n-1)e \right) - \left( \frac{n(n-1)}{2} ( -\ell ) + (n-1) f + t \right) \\
 = \ (n-1) ( - e - f + n \ell ) - t \ = \ t (2n-3) \end{multline*}
since $-e-f+n\ell = 2t$. But $\cT^{2t} ( F^* \otimes L )$ is irreducible of dimension $t (2n - 3)$ by Lemma \ref{TypeTParameterSpace}. Therefore, $\pi_*|_X$ is dominant. \end{proof}

In view of conditions (ii) and (iii), by (\ref{syzygy}) the locus $\QFo$ is nonempty only if $t \ge h^1 ( C, L^{-1} \otimes \wedge^2 F )$. In fact we shall later require the strict inequality
\begin{equation} t \ = \ \frac{1}{2} ( - e - f + n\ell) \ > \ h^1 ( C, L^{-1} \otimes \wedge^2 F ) . \label{BoundOnT} \end{equation}
A computation using Riemann--Roch shows that this is equivalent to
\[ e \ \le \ (2n-3)f - 2 \cdot h^0 ( C, L^{-1} \otimes F ) - n(n - 2) \ell - n(n-1)(g-1) - 1. \]
\begin{definition} \label{defneVfd} We set
\[ \eVfd \ := \ (2n - 3) f - n(n - 2) \ell - n(n-1)(g-1) - 2 \cdot h^0 ( C, L^{-1} \otimes \wedge^2 V ) - 1 . \]
As $h^0 ( C, L^{-1} \otimes \wedge^2 F ) \le h^0 ( C, L^{-1} \otimes \wedge^2 V )$ for any subbundle $F \subset V$, if $e \le \eVfd$ then the condition (\ref{BoundOnT}) is satisfied for \emph{any} $F \in \IQo_f (V)_\dpr$. \end{definition}

As the proof of the next proposition is rather involved, we postpone it to {\S} \ref{QFoDenseProof}.

\begin{proposition} \label{QFoDense} Suppose $e \le \min\{ \eVfd , f_\delta \}$ and $-e-f+n\ell \equiv 0 \mod 2$. 
 Then $\QFo$ is nonempty and dense in $\QF$. In particular, $\QF$ is contained in $\bIQeo$. \end{proposition}

\begin{proposition} If $e \le \eVfd$, then $\QF$ is nonempty and irreducible. \label{QFirr} \end{proposition}

\begin{proof} By Proposition \ref{QFoDense}, it suffices to show that $\QFo$ is irreducible. If $X_1$ and $X_2$ were distinct irreducible components of $Q_F^\circ$, then by Proposition \ref{dominant}, the restriction of $\pi_*$ to either component would be dominant with irreducible fibers. But then $X_1$ and $X_2$ would have to intersect along a dense subset of a generic fiber, contradicting the smoothness of $\QFo$ proven in Proposition \ref{QFoSmooth}. Thus $Q_F^\circ$ is irreducible. \end{proof}

\subsection{Proof of Theorem \ref{Even}}

Let $e$ be an integer satisfying $-e-f_\dpr + n\ell \equiv 0 \mod 2$.

\begin{lemma} The condition $-e-f_\dpr + n\ell \equiv 0 \mod 2$ implies that $e \equiv f_\delta \mod 2$. \label{parityefd} \end{lemma}

\begin{proof} We distinguish three cases. If both $n$ and $\ell$ are odd, then $\dpr \ne \delta$ by Definition \ref{deltapr}. Thus $f_\dpr \equiv f_\delta + 1 \mod 2$ by Theorem \ref{ParityComponents} (b). Then
\[ e - f_\delta \ \equiv \ -e - f_\dpr - 1 \ \equiv -e - f_\dpr + n \ell \ \mod 2 . \]
But $-e -f_\dpr + n \ell$ is even, so $e \equiv f_\delta \mod 2$.

If $\ell$ is even, then $f_\dpr \equiv f_\delta \mod 2$ by Theorem \ref{ParityComponents} (a), and the statement follows. 

Lastly, if $n$ is even, then $\dpr = \delta$, so $f_\delta = f_\dpr$, and the statement follows. \end{proof}

\begin{proposition} \label{covering} Suppose $e \le e( V, f_\dpr, \delta)$ (cf.\ (\ref{defneVfd})) and $- e - f_\dpr + n\ell \equiv 0 \mod 2$. 
\begin{enumerate}
\item[(a)] The collection $\QF : F \in \IQo_{f_\dpr} (V)_{\dpr}$ is an open covering of $\bIQeo_\delta$.
\item[(b)] If furthermore $e \le f_\delta$, then any two members of the covering have nonempty intersection.
\end{enumerate}
\end{proposition}

\begin{proof} Since $e \le e( V, f_\dpr, \delta)$, by Proposition \ref{QFoDense} all the loci $\QF$ are contained in $\bIQeo_\delta$. By Lemma \ref{FirstPptiesQF} (a), each $\QF$ is open in $\bIQeo_\delta$. 
 Suppose $E \in \bIQeo_\delta$. Then $\bE / E$ is of type $\cT$ by Proposition \ref{Smoothable}; and by Lemma \ref{evdom}, we can find $F \in \IQo_{f_\dpr} (V)_\dpr$ intersecting $E$ in zero at some point, and hence at the generic point. Thus $E \in \QF$. This proves part (a).

As for (b): Suppose $F$ and $F'$ are elements of $\IQo_{f_\dpr} (V)_\dpr$. Then a generic $\Lambda \in \OG(V)_\delta$ intersects the fibers of both $F$ and $F'$ in zero. By Lemma \ref{evdom}, we can find a subbundle $E_1 \in \IQo_{f_\delta} (V)_\delta$ intersecting $F$ and $F'$ in rank zero. As by hypothesis $e \le f_\delta$, in view of Lemma \ref{parityefd} we may choose a type $\cT$ elementary transformation $E \subset E_1$ with $\deg (E) = e$. Then $[E \to V]$ belongs to $\QF \cap Q_{F'}^e$. \end{proof}

Set $e (V) := \min \{ f_1, f_2, e(V, f_1, 1), e (V, f_2, 2) \}$. Now we can prove the main theorem.


\begin{proof}[Proof of Theorem \ref{Even}] Suppose $e \le e (V)$. We shall show that any $\bIQeo_\delta$ which is nonempty is irreducible. Suppose $X_1$ and $X_2$ are nonempty irreducible components of $\bIQeo_\delta$. By Proposition \ref{covering} (b), there exist $F_1, F_2 \in \IQo_{f_\dpr} (V)_\dpr$ such that for $1 \le i \le 2$, the locus $Q_{F_i}^e$ intersects $X_i$ in an open, therefore dense subset of $X_i$. Since $Q_{F_i}^e$ is irreducible by Proposition \ref{QFirr}, it is dense in each $X_i$ containing it. Moreover, by Proposition \ref{covering} (a), the open set $Q_{F_1}^e \cap Q_{F_2}^e$ is nonempty. Therefore it is dense in each $Q_{F_i}^e$, thus also in each $X_i$. Hence
\[ X_1 \ = \ \overline{X_1} \ = \ \overline{Q_{F_1}^e} \ = \ \overline{Q_{F_1}^e \cap Q_{F_2}^e} \ = \ \overline{Q_{F_2}^e} \ = \ \overline{X_2} \ = \ X_2 . \]
Therefore, $\bIQeo_\delta$ is irreducible.

If $\ell$ is odd, then by Theorem \ref{ParityComponents} (b) and Definition \ref{labeling}, the locus $\IQeo_\delta$ is nonempty only if $e \equiv \delta \mod 2$. Therefore, $\overline{\IQeo} = \bIQeo_\delta$ is irreducible.

If on the other hand $\ell$ is even, then for all $x \in C$, the respective images of $\bIQeo_1$ and $\bIQeo_2$ via any evaluation map $\ev_x \colon \IQe \dashrightarrow \OG(V|_x)$ are disjoint. Hence $\overline{\IQeo_1}$ and $\overline{\IQeo_2}$ are disjoint and irreducible. \end{proof}

\noindent This proves Theorem \ref{TypeT} (b) in the even rank case, and we obtain the following characterization of the natural compactification $\overline{\IQeo}$ of $\IQeo$.

\begin{corollary} \label{compactification} Suppose $e \le e (V)$. Then $E \in \IQe$ belongs to $\overline{\IQeo}$ if and only if $\bE / E$ is of type $\cT$. \end{corollary}

\subsection{Proof of Proposition \ref{QFoDense}} \label{QFoDenseProof}

To ease notation, we write $Y := \Gr(2, F)$ and $\hY$ for the cone over $\Gr( 2, F)$ in the vector bundle $L^{-1} \otimes \wedge^2 F$. Recall that a finite set $y_1 , \ldots , y_t \in \cc^{N + 1}$ (resp., $\pp^N$) is said to be \textsl{in general position} if for $1 \le t' \le t$, the span of any $t'$ of the $y_i$ has dimension $\min \{ t' , N + 1 \}$ (resp., $\min \{ t' - 1, N \}$). We recall now a special case of \cite[Definition 4.20]{CCH1}.

\begin{definition} \label{genYvalued} A principal part $p \in \Prin ( L^{-1} \otimes \wedge^2 F )$ will be called \textsl{general $\hY$-valued} if the following conditions are satisfied.
\begin{itemize}
\item $p$ can be represented by $\sum_{i=1}^t \frac{\sigma_i}{z_i}$ where $z_1 , \ldots , z_t$ are local parameters at distinct points $x_1 , \ldots , x_t$ of $C$ respectively, and $\sigma_i$ is a section of $\hY \subset L^{-1} \otimes \wedge^2 F$ near $x_i$.
\item The cohomology classes $\left[ \frac{\sigma_i}{z_i} \right]$ are in general position in $H^1 ( C, L^{-1} \otimes \wedge^2 F )$, when this space is nonzero.
\end{itemize}
\end{definition}
\noindent (Recall that by Lemma \ref{alternativePsi}, the class $\left[ \frac{\sigma_i}{z_i} \right]$ defines the image of the point $\sigma_i (x_i)$ by the map $\psi \colon \Gr (2, F)  \dashrightarrow \pp H^1 ( C, L^{-1} \otimes \wedge^2 F )$.)

\begin{proof}[Proof of Proposition \ref{QFoDense}] By Lemma \ref{QFnonempty} and Lemma \ref{parityefd}, the locus $\QF$ is nonempty. Suppose that $E$ is a point of $\QF \setminus \QFo$. We shall prove the proposition by showing that there exists a deformation $\{ \cE'_s : s \in \Delta \}$ of $E$ in $\QF$ parameterized by a disk $\Delta$ such that $\cE'_0 = E$ while $\cE'_s \in \QFo$ for $s \ne 0$.

The saturation $\bE$ of $E$ is an isotropic subbundle of $V$, of degree $\be \ge e$. By Lemma \ref{GoodCoords}, we may assume that $V$ is an extension of the form $V^p$ for an antisymmetric principal part $p$ satisfying $\bE = \Gamma_0 \cap V^p \cong \Ker ( p )$. Note that $\deg(p) = -\be - f + n\ell$, which by Lemma \ref{lemmaTypeT} is an even integer $2 \bt$.

By the proof of \cite[Lemma 2.4]{CH3} (essentially a diagonalization procedure; see also \cite[Lemma 2.7]{CH2}), the principal part $p$ can be represented by a sum
\[ \sum_{i = 1}^m \frac{\sigma_i}{z_i^{d_i}} \]
where $z_i$ is a local coordinate at a point $x_i \in C$, and $\sigma_i$ is a section of $L^{-1} \otimes \wedge^2 F$ near $x_i$ which is decomposable at every point of its domain. We have
\[ \sum_{i=1}^m d_i \ = \ \frac{1}{2} \left( \deg ( F^* \otimes L) - \deg ( \bE ) \right) \ = \ \bt , \]
where $\bt$ is as defined above. (Note that the $x_i$ need not be distinct.)
%
 Then by \cite[Lemma 4.21]{CCH1} and its proof, there exists an analytic family of principal parts $\{ p_s : s \in \Delta \}$ parameterized by an open disk $\Delta$ around $0 \in \cc$, such that $p_0 = p$, while for $s \ne 0$ we have
\[ p_s \ = \ \sum_{i=1}^m \sum_{j=1}^{d_i} \frac{\rho_{i, j}}{s} \frac{\sigma_i}{(z_i - s \tau_{i, j})} \]
for suitable nonzero scalars $\rho_{i, j}$ and $\tau_{i,j}$; and moreover $p_s$ is general $\hY$-valued in the sense of Definition \ref{genYvalued} for $s \ne 0$.

We claim that if $\be \ne e$, then we may assume that $E$ is a \emph{general} type $\cT$ elementary transformation of $\bE$. For; $\cT^{\be - e} ( \bE )$ is completely contained in $\QF$, since $\rank (\bE \cap F) = 0$. As $\cT^{\be - e} ( \bE )$ is irreducible by Lemma \ref{TypeTParameterSpace} (a), if a general point belongs to the closure of $\QFo$ in $\QF$, then in fact every point does.

Therefore, by Lemma \ref{TypeTParameterSpace} (b) we may assume that
\[ \bE / E \ \cong \ \bigoplus_{k=1}^{\frac{1}{2}(\be - e)} \cO_{u_k} \otimes \cc^2 , \]
for distinct points $u_k \in C$ lying outside $\Supp ( p )$. Notice that $\frac{1}{2} (\be - e) = t - \bt$. Thus for each $k$ there exists a local coordinate $w_k$ centered at $u_k$ and a decomposable section $\nu_k$ of $L^{-1} \otimes \wedge^2 F$ near $u_k$, such that 
\[ E \ = \ \bE \cap \Ker \left( \sum_{k=1}^{t - \bt} \frac{\nu_k}{w_k} \right) \ = \ \Ker \left( p_0 + \sum_{k=1}^{t - \bt} \frac{\nu_k}{w_k} \right) , \]
where as usual we view the principal parts as maps $F^* \otimes L \to \sPrin ( F )$.

Now by Definition \ref{defneVfd} of $e ( V, f_\dpr, \delta )$, we have the inequality
\[ \frac{1}{2} ( -e - f_\dpr + n\ell ) \ = \ t \ \ge \ h^1 ( C, L^{-1} \otimes \wedge^2 F ) + 1 . \]
Thus by \cite[Lemma 4.23]{CCH1},
 there exist nowhere zero analytic functions $a_{i, j} (s)$ and $b_k (s)$ on $\Delta$ such that the family of principal parts
\[ p'_s \ := \ p_s + \sum_{i=1}^m \sum_{j=1}^{d_i} s \cdot a_{i, j}(s) \cdot \frac{\sigma_i}{z_i - s \tau_{i, j}} + \sum_{k=1}^r s \cdot b_k (s) \cdot \frac{\nu_k}{w_k} \]
satisfies
\begin{equation} [ p'_s ] \ \equiv \ [ p ] \ = \ \delta (V) \ \hbox{ for all } s \in \Delta, \label{equivdeltaV} \end{equation}
and for $s \ne 0$, the principal part $p_s'$ is general $\hY$-valued in the sense of Definition \ref{genYvalued}.

We consider the family $\{ \cE_s : s \in \Delta \}$ of elementary transformations of $F^* \otimes L$ given by
\begin{equation} \cE_s \ = \ \Ker ( p_s ) \ \subset \ F^* \otimes L . \label{EsKerps} \end{equation}
For $s \ne 0$, the sheaf $\cE_s$ has degree $-e$ and $p_s$ is general $\hY$-valued. The latter condition clearly implies that $\cE_s$ satisfies property (ii), and furthermore, by the statement of general position and since $-e-f_{\dpr} + n \ell > h^1 ( C, V )$, in particular the cohomology classes
\[ \left[ \frac{\sigma_i}{z_i - s \lambda_{i, j}} \right] : \quad 1 \le i \le m; \quad 1 \le j \le d_i \quad \hbox{and} \quad \left[ \frac{\nu_k}{w_k} \right] : 1 \le k \le \bt \]
span $H^1 ( C, L^{-1} \otimes \wedge^2 F )$ for $s \ne 0$. Thus $\cE_s$ also has property (iii) for $s \ne 0$.

Following \cite[{\S} 4.3.1]{CCH1}, we globalize the construction (\ref{V}) to a family of extensions $\{ \cV_s : s \in \Delta \}$ by
\[ \cV_s (U) \ := \ V^{p_s} (U) \ = \ \{ ( f, \phi ) \in \sRat ( F ) (U) \oplus (F^* \otimes L) (U) : p_s ( \phi ) = \overline{f} \} \]
for each open set $U \subseteq C$. Since each $p_s$ is antisymmetric, as before (\ref{standardquadform}) restricts to an orthogonal structure on each $\cV_{p_s}$. 
 By (\ref{equivdeltaV}), in fact $\cV_s \cong V$ for all $s \in \Delta$. By the description (\ref{EsKerps}) and Proposition \ref{cohomlifting} (c), 
 the family $\cE \subset \pi_C^* (F^* \otimes L)$ defines a family of saturated subsheaves of $V$. As $\cE_s \subset \sRat (F^* \otimes L)$ and the latter is isotropic with respect to the form (\ref{standardquadform}), in particular $\cE_s$ is isotropic.

Now $\cE$ is only flat over $\Delta \setminus \{ 0 \}$, as $\Ker ( p_s )$ has degree $-e$ for $s \ne 0$ but $\cE_0 = \bE$ has degree $-\be$. Write $\cE_s'$ for the flat family in $\IQeo$ satisfying $\cE'_s = \cE_s$ for $s \ne 0$. The flat limit $\cE_0'$ in $\IQeo$ is a full rank subsheaf of $\bE = \Ker (p)$. We will have finished if we can show that $\cE_0' = E$. Now for $s \ne 0$, each $\cE_s' = \cE_s$ is contained in
\[ \Ker \left( \sum_{k=1}^{\frac{1}{2} (e - \be)} s \cdot \frac{\nu_k}{w_k} \right) . \]
As an element of $\cT^{2 \bt} ( F^* \otimes L )$, this is independent of $s$ for $s \ne 0$, being simply equal to
\begin{equation} \Ker \left( \sum_{k=1}^{\bt} \frac{\nu_k}{w_k} \right) . \label{constantET} \end{equation}
Thus the limit $\cE_0'$ is also contained in (\ref{constantET}). Hence $\cE_0'$ is contained in
\[ \Ker \left( \sum_{k=1}^{\bt} \frac{\nu_k}{w_k} \right) \cap \Ker ( p ) , \]
which is exactly $E$. As $\deg ( \cE_0' ) = -e$ by flatness, we have $\cE_0' = E$. \end{proof}

\section{Isotropic Quot schemes of odd rank orthogonal bundles} \label{OddRank}

Let $V \cong V^* \otimes L$ be an orthogonal bundle of rank $2n+1 \ge 3$. As noted in {\S} \ref{determinants}, the degree $\deg(L)$ is even, say $2m$. Therefore, $\det(V) = L^n \otimes M$ for some square root $M$ of $L$, where $\deg (M) = m$. 
Now $M \cong \Hom (M, L)$ is an $L$-valued orthogonal line bundle, and the orthogonal direct sum $V \perp M$ is an $L$-valued orthogonal bundle of rank $2n+2$ and determinant $L^n \otimes M \otimes M = L^{n+1}$. Thus $V \perp M$ has isotropic subbundles of rank $n+1$ by Theorem \ref{EvenRankExistence}.

\subsection{Saturated loci in odd rank}

We give now a generalization of \cite[Proposition 3.11 (3)]{CH4}. Note that $\IQ_{e + m} (V)$ and $\IQ_e (V \perp M )$ are subloci of, respectively, $\Quot_{n, e+m} (V)$ and $\Quot_{n+1, e} ( V \oplus M )$.

\begin{lemma} Each component of ${\IQ^\circ_{e+m} ( V \perp M )}$ is canonically isomorphic to $\IQeo$. \label{EvenToOdd} \end{lemma}

\begin{proof} For any rank $n+1$ isotropic subbundle $\tE \subset V \perp M$, the intersection $\tE \cap V$ is isotropic in $V$. As an isotropic subspace of a fiber $V|_x$ has dimension at most $n$, in fact $\dim ( E \cap V|_x ) = n$ for all $x \in C$. Therefore, we have a diagram of maps of vector bundles
\[ \xymatrix{ 0 \ar[r] & V \ar[r] & V \perp M \ar[r] & M \ar[r] & 0 \\
 0 \ar[r] & \tE \cap V \ar[r] \ar[u] & \tE \ar[r] \ar[u] & M \ar[r] \ar[u]^= & 0 , } \]
Thus the association $\tE \mapsto \tE \cap V$ defines a morphism $ {\IQ^\circ_{e+m} (V \perp M)} \to \IQeo$.

Now suppose $E \in \IQeo$. By Lemma \ref{EvenRankExistence}, we can find an orthonormal frame for $V \perp M$ over some Zariski open subset $U \subseteq C$. By \cite[Lemma 3.10 (3)]{CH4}, the bundle $E|_U$ can be completed to an isotropic subbundle of $(V \perp M)|_U$ in exactly two ways. As $C$ is a curve, each such completion admits a unique extension to a rank $n+1$ isotropic subbundle $\tE$ of $V \perp M$, clearly satisfying $\tE \cap V = E$. 
 We denote these completions by $\tE_1$ and $\tE_2$. By the argument in the preceding paragraph, each $\tE_j$ is an extension $0 \to E \to \tE_j \to M \to 0$, so defines a point of $\IQ^\circ_{e+m} ( V \perp M )$. Moreover, as
\[ \rank ( \tE_1 \cap \tE_2 ) \ = \ \rank (E) \ = \ n \ \not\equiv \ (n + 1) \mod 2 , \]
by Lemma \ref{GH} and Theorem \ref{Even} (a) the points $[ \tE_1 \to V \perp M ]$ and $[ \tE_2 \to V \perp M ]$ belong to opposite components of $\IQ^\circ_e ( V \perp M )$. We conclude that the map $\IQ^\circ_{e+m} ( V \perp M )_\delta \to \IQeo$ is bijective for each $\delta \in \{ 1 , 2 \}$. \end{proof}

Before proving Theorem \ref{Odd}, we recall for convenience the definition of the Stiefel--Whitney class $w(V)$ for orthogonal bundles $V$ of rank $2n+1$. In this case, $\ell = 2m$ is even, and by Lemma \ref{twistOdd} there exists $M$ such that $V \otimes M^{-1}$ is $\Oc$-valued orthogonal of trivial determinant. Then 
\[ w(V) \ := \ w_2 (V \otimes M^{-1}) + nm \quad \text{in} \quad H^2 (C, \Z_2 ) \ \cong \ \Z_2. \]

\begin{proof}[Proof of Theorem \ref{Odd}] Suppose $V \cong V^* \otimes L$ is orthogonal of rank $2n+1$ and of determinant $L^n \otimes M$. As above, $V \perp M$ is $L$-valued orthogonal of rank $2(n+1)$ and determinant $L^{n+1}$. By Theorem \ref{Even} (a) and Lemma \ref{EvenToOdd}, for $e \le e ( V \perp M ) - m$ and $e \equiv w ( V \perp M ) - m \mod 2$, the locus $\IQeo$ and hence also $\overline{\IQeo}$ is nonempty, irreducible and generically smooth of the expected dimension. But we have 
\[ w ( V \perp M ) - m \ = \ w_2 ( (V \perp M) \otimes M^{-1} ) + (n+1)m - m \ = \ w_2 ( V \otimes M^{-1} ) + nm \ = \ w(V) . \]
Setting $e(V) = e ( V \perp M ) - m$, we are done. \end{proof}

\subsection{Nonsaturated loci in odd rank} \label{OddRankComponents}

Although we will not give the same level of detail as in the even rank case, we make one observation for a general orthogonal $V$ of odd rank. We continue to assume $L$ is a line bundle of even degree $2m$.

\begin{lemma} Let $V$ be an $L$-valued orthogonal bundle of rank $2n+1$ which is general in its component of moduli. Then for all $e$, the saturated locus $\IQeo$ has the expected dimension $I ( n+1, 2m, e - m )$. \label{VeryStableOddRank} \end{lemma}

\begin{proof} By Lemma \ref{EvenToOdd}, it suffices to show that $h^0 (C, \Kc L^{-1} \otimes \wedge^2 E^\perp ) = 0$ for all $E \in \IQeo$. One shows that this is satisfied by all very stable $V$, by an argument similar to that in the proof of Lemma \ref{VeryStableSmooth}. We omit the details. \end{proof}

\begin{proposition} Let $V$ be a general $L$-valued orthogonal bundle of rank $2n+1$. Then $\IQe$ is equidimensional of dimension $I \left( n+1, 2m , e - m \right)$. \end{proposition}

\begin{proof} Let $Y$ be any nonempty irreducible component of $\IQe$. By generality and by Lemma \ref{VeryStableOddRank}, we may assume that for all $e$, the saturated locus $\IQeo$ has the expected dimension $I \left( n+1, 2m , e - m \right)$. Let $\be$ be the degree of the saturation of a general element of $Y$. More precisely, let $\be$ be the unique integer such that the locally closed subset
\[ Y^\circ \ := \ \{ [ E \to V ] \in Y : \deg ( \bE ) \ = \ \be \} \]
is open and dense in $Y$. Now $Y^\circ$ is a fiber bundle over a component of $\IQ_{\be}^\circ (V)$, with fiber $\Elm^{\be-e}(\bE)$ over $\bE$. Thus $\dim Y = \dim Y^\circ = \dim \IQ_{\be}^\circ (V) + \dim \Elm^{\be-e}(\bE)$. By the hypothesis of expected dimension, we obtain
\[ \dim Y \ = \ I \left( n+1, 2m , \be - m \right) + n ( \be - e ) \ = \ I \left( n+1, 2m , e - m \right) . \qedhere \]
\end{proof}

\bibliography{orthQuot_bib}

\bibliographystyle{abbrv}

\end{document}